\documentclass[12pt,reqno]{amsart}
\usepackage{fullpage}

\usepackage{amsfonts,amscd}
\usepackage{amssymb}
\usepackage{url}
\usepackage[english]{babel}

\theoremstyle{plain}
\newtheorem{theorem}                {Theorem}      [section]
\newtheorem{proposition}  [theorem]  {Proposition}

\newtheorem{lemma}        [theorem]  {Lemma}

\theoremstyle{definition}
\newtheorem{example}      [theorem]  {Example}
\newtheorem{remark}       [theorem]  {Remark}

\newtheorem{problem}        [theorem]  {Problem}

\newcommand{\secref}[1]{\S\ref{#1}}
\numberwithin{equation}{section}

\def \R{{\mathbb R}}
\def \s{{\mathbb S}}

\def \C{{\mathbb C}}
\def \H{{\mathbb H}}

\usepackage{color}

\DeclareMathOperator{\grad}{grad}

\DeclareMathOperator{\Div}{div}
\DeclareMathOperator{\ricci}{Ricci}

\numberwithin{equation}{section}

\begin{document}

\title[Reduction methods for the bienergy]{Reduction methods for the bienergy}

\author{S.~Montaldo}
\address{Universit\`a degli Studi di Cagliari\\
Dipartimento di Matematica e Informatica\\
Via Ospedale 72\\
09124 Cagliari, Italia}
\email{montaldo@unica.it}

\author{C.~Oniciuc}
\address{Faculty of Mathematics\\ ``Al.I. Cuza'' University of Iasi\\
Bd. Carol I no. 11 \\
700506 Iasi, ROMANIA}
\email{oniciucc@uaic.ro}

\author{A.~Ratto}
\address{Universit\`a degli Studi di Cagliari\\
Dipartimento di Matematica e Informatica\\
Viale Merello 93\\
09123 Cagliari, Italia}
\email{rattoa@unica.it}

\begin{abstract}
This paper, in which we develop ideas introduced in \cite{MR}, focuses on \emph{reduction methods} (basically, group actions or, more generally, simmetries) for the bienergy. This type of techniques enable us to produce examples of critical points  of the bienergy by reducing the study of the relevant fourth order PDE's system to ODE's. In particular, we shall study rotationally symmetric biharmonic conformal diffeomorphisms between \emph{models}. Next, we will adapt the reduction method to study an ample class of $G-$invariant immersions into the Euclidean space. At present, the known instances in these contexts are far from reaching the depth and variety of their companions which have provided fundamental solutions to classical problems in the theories of harmonic maps and minimal immersions. However, we think that these examples represent an important starting point which can inspire further research on biharmonicity. In this order of ideas, we end this paper with a discussion of some open problems and possible directions for further developments.
\end{abstract}

\subjclass[2010]{53A07, 53C42, 58E20}

\keywords{Biharmonic maps, biharmonic immersions, transformation groups, equivariant differential geometry, cohomogeneity one hypersurfaces, models}

\thanks{The second author was supported by a grant of the Romanian National Authority for Scientific Research and Innovation, CNCS -- UEFISCDI, project number PN-II-RU-TE-2014-4-0004}

\maketitle

\section{Introduction}\label{intro}
According to B.-Y. Chen \cite{Chen} an isometric immersion $\varphi:M^m\hookrightarrow\R^n$ is {\it biharmonic} if
\begin{equation}\label{def-chen}
\Delta H =(\Delta H_1,\ldots, \Delta H_n)=0 \,\, ,
\end{equation}
where $H=(H_1,\ldots,H_n)$ is the \emph{mean curvature vector field} and $\Delta$ denotes the Beltrami-Laplace operator on $M$ (our sign convention is such that $\Delta h=-h''$ when $h$ is a function of one real variable).
It follows from Beltrami's equation
$$
m\, H=-(\Delta\, \varphi_1,\ldots,\Delta\, \varphi_n)
$$
that the biharmonicity condition is equivalent to
$$
\Delta^{2}\, \varphi=(\Delta^{2}\, \varphi_1,\ldots,\Delta^{2}\, \varphi_n)=0\,\,,
$$
a fact which justifies Chen's definition of biharmonic immersions.
The study of biharmonic immersions in $\R^n$ can be set in a more general variational context in Riemannian geometry. To the purpose of describing this setting, we first recall that
{\it harmonic maps} (\cite{ES}) are the critical points of the {\em energy} functional
\begin{equation}\label{energia}
E(\varphi,D)=\frac{1}{2}\int_{D}\,|d\varphi|^2\,dv_g \,\, ,
\end{equation}
for any compact domain $D$ with smooth boundary of an orientable manifold $M$,
where $\varphi:(M,g)\to(N,h)$ is a smooth map between two Riemannian
manifolds. In analytical terms, the condition of harmonicity is equivalent to the fact that the map $\varphi$ is a solution on $M$ of the Euler-Lagrange equation associated to the energy functional \eqref{energia}, i.e.
\begin{equation}\label{harmonicityequation}
    {\rm trace} \, \nabla d \varphi =0 \,\, .
\end{equation}
The left member of \eqref{harmonicityequation} is a vector field along the map $\varphi$, or, equivalently, a section of the pull-back bundle $\varphi^{-1} \, (TN)$: it is called {\em tension field} and denoted $\tau (\varphi)$. In a local chart the tension field is given by
\begin{equation}\label{tensionfield}
    \tau^\gamma (\varphi)= g^{ij} \left ( \nabla(d \varphi) \right )_{ij}^\gamma \,\, ,
\end{equation}
where the explicit expression for the second fundamental form is
\begin{equation}\label{secondaformalocale}
    \left ( \nabla(d \varphi) \right )_{ij}^\gamma = \frac{\partial^2 \varphi^\gamma}{\partial x^i \, \partial x^j} - {}^M\Gamma_{ij}^k \, \frac{\partial \varphi^\gamma}{\partial x^k} \, + \, {}^N\Gamma_{\alpha \beta}^\gamma \, \frac{\partial \varphi^\alpha}{\partial x^i} \, \frac{\partial \varphi^\beta}{\partial x^j} \,\, .
\end{equation}
In \eqref{secondaformalocale}, ${}^M\Gamma$ and ${}^N\Gamma$ denote the Christoffel symbols of the Levi-Civita connections of $(M,g)$ and $(N,h)$ respectively. Also, note that Einstein's convention of sum over repeated indices is adopted. In particular, inspection of \eqref{secondaformalocale} reveals that the harmonicity equation \eqref{harmonicityequation} is a second order, elliptic system of partial differential equations.
Every vector field ${\mathcal V}$ along $\varphi$, supported in $D$,  defines a variation of $\varphi$, supported in $D$, by setting
\begin{equation}\label{variation}
    \varphi_t (p)= \exp_{\varphi(p)}\left( t\,{\mathcal V}(p)\right) \,\, , \qquad t\in(-\varepsilon ,\varepsilon)
\end{equation}
(note that $\varphi_0=\varphi$). Then, for such vector field ${\mathcal V}$, we have:
\begin{equation}\label{decrescenzarapidaenergia}
  \nabla_{\mathcal V} \, E(\varphi)= \left . \frac{d\,E(\varphi_t) }{dt}\right |_{t=0} =-\, \int_{D}\,\langle\tau(\varphi) ,\, {\mathcal V}\rangle\,dv_g \,\, .
\end{equation}
In particular, it follows that \eqref{harmonicityequation} is equivalent to the vanishing of the directional derivative in \eqref{decrescenzarapidaenergia} for all ${\mathcal V}$. Moreover, we point out that \eqref{decrescenzarapidaenergia} implies the fact that the tension field $\tau (\varphi)$ provides the direction in which the energy decreases more rapidly.

The study of harmonic maps is a very wide area of research, involving a rich interplay of geometry,
analysis and topology. We refer to \cite{PBJCW,JEELLE1,JEELLE2,Xin} for notation and
background on harmonic maps and to \cite{BMBib} for a more recent bibliography.

Now, we are in the right position to introduce the topic which is the main object of this paper, i.e., the notion of a {\it biharmonic map}. More precisely, biharmonic maps, which provide a natural generalisation of harmonic maps, are the critical points of the bienergy functional (as suggested by Eells--Lemaire \cite{EL83})
\begin{equation}\label{bienergia}
    E_2(\varphi,D)=\frac{1}{2}\int_{D}\,|\tau (\varphi)|^2\,dv_g \,\, .
\end{equation}
In \cite{Jiang} Jiang derived the first variation and the second variation formulas for the bienergy. In particular, he showed that the Euler-Lagrange equation associated to $E_2(\varphi)$ is
\begin{equation}\label{bitensionfield}
    \tau_2(\varphi) = - J\left (\tau(\varphi) \right ) = - \triangle \tau(\varphi)- \rm{trace} R^N(d \varphi, \tau(\varphi)) d \varphi = 0 \,\, ,
    \end{equation}
where $J$ is the Jacobi operator of $\varphi$, $\triangle$ is the rough Laplacian defined on sections of $\varphi^{-1} \, (TN)$ and
\begin{equation}\label{curvatura}
    R^N (X,Y)= \nabla_X \nabla_Y - \nabla_Y \nabla_X -\nabla_{[X,Y]}
\end{equation}
is the curvature operator on $(N,h)$.

In this context, the biharmonic version of \eqref{decrescenzarapidaenergia} is (see\cite{Jiang})
\begin{equation}\label{decrescenzarapidabienergia}
  \nabla_{\mathcal V} \, E_2(\varphi)= \left . \frac{d\,E_2(\varphi_t) }{dt}\right |_{t=0} = \int_{D}\,\langle  \tau_2(\varphi) ,\,{\mathcal V}\rangle \,dv_g \,\, .
\end{equation}
Therefore, \eqref{bitensionfield} represents the vanishing of the directional derivative in \eqref{decrescenzarapidabienergia} for all ${\mathcal V}$. In particular, the bitension field $\tau_2 (\varphi)$ provides the direction in which the bienergy varies more rapidly.

We point out that \eqref{bitensionfield} is a {\it fourth order} semi-linear elliptic system of differential equations. Fourth order differential equations are of great importance in various fields. By way of example, we cite an instance which is well-known to civil engineers: the structural problem of a beam resting on an elastic soil is amenable to the following differential equation, in which the unknown function $Y(x,t)$ represents the response of the beam in the position $x$ at the time $t$:
\begin{equation}\label{trave}
    \frac{\partial^2}{\partial x^2} \left( EI(x) \, \frac{\partial^2 \, Y }{\partial x^2}\right) + m \, \frac{\partial^2 \, Y}{\partial t^2} + k \, Y= f(x,t) \,\, ,
\end{equation}
where $EI(x)$ measures the flexural rigidity (Young modulus), $m$ and $k$ are two constants depending respectively on the material of the beam and on the elasticity of the soil; the right member $f(x,t)$ represents the external load (see \cite{Corradi} for details).

First, let us note that any harmonic map is an absolute minimum of the bienergy, and so it is trivially biharmonic. Therefore, a general working plan is to study the existence of those biharmonic maps which are not harmonic: these will be refered to as {\it proper biharmonic maps}.
We refer to \cite{Montaldo} for an introduction to biharmonicity, existence results and general properties of biharmonic maps. \\

Now, we can make explicit the generalization of Chen's definition of a biharmonic submanifold into the Euclidean space. Indeed, a submanifold into a Riemannian manifold $(N,h)$, that is an isometric immersion $\varphi:M\hookrightarrow (N,h)$,  is called a {\it biharmonic submanifold} if $\varphi$ is a biharmonic map.  In particular, minimal submanifolds are trivially biharmonic, so that we call {\it proper} biharmonic any biharmonic submanifold which is not minimal. We observe that when $N=\R^n$ the curvature term in \eqref{bitensionfield} vanishes and the equation $\tau_2(\varphi) =0$ is equivalent to \eqref{def-chen}. Therefore, the definition of biharmonic  submanifolds extends the original definition of Chen. \\

In the important, special case of hypersurfaces, biharmonicity can be expressed by means of the following general result (see \cite{BMO13,C84,LM08,O02,O10}):

\begin{theorem}\label{th: bih subm N}
Let  $\varphi:M^{n-1}\hookrightarrow N^{n}$ be an isometric immersion with mean curvature vector field $H=(f/(n-1))\,\eta$. Then $\varphi$ is biharmonic if and only if the normal and the tangent components of $\tau_2(\varphi)$  vanish,  i.e.,
\begin{subequations}
\begin{equation}\label{eq: caract_bih_normal}
\Delta {f}+f |A|^2- f\ricci^N(\eta,\eta)=0
\end{equation}
and
\begin{eqnarray}\label{eq: caract_bih_tangent}
2 A(\grad f)+  f \grad f-2 f \ricci^N(\eta)^{\top}=0
\end{eqnarray}
\end{subequations}
respectively, where $A$ is the shape operator and $\ricci^N(\eta)^{\top}$ is the tangent  component of the Ricci tensor field of $N$ in the direction of the unit normal vector field $\eta$ of $M$ in $N$.
\end{theorem}

As a general fact, we have to say that we encounter enormous technical difficulties to study \emph{fourth order} differential equations. In particular, the presently known examples of biharmonic maps have essentially been obtained by means of geometric intuition and simplification (see, for instance, \cite{BMO1,BMO2,MOR-JJAA}).
In this order of ideas, {\it equivariant theory} (or, {\it reduction theory}) deals with special families of maps having enough symmetries to guarantee that the original PDE's problem reduces to the study of certain ODE's. In particular, the study of certain harmonic maps may reduce to that of a certain second order {\it ordinary} differential equation (we refer to \cite{ER, Xin} for background and examples), while the search for $G-$invariant, cohomogeneity one minimal submanifolds becomes equivalent to proving the existence of suitable solutions to a second order ODE's system (see \cite{Hsiang}, for instance).

In this paper we shall illustrate how the reduction framework can be adapted to include the study of biharmonic maps. Although the study of the resulting
fourth order ordinary differential equations remains, in general, a very difficult task, one of our aims is to provide analysts with a direct approach
which permits them to compute the relevant equations avoiding the heavy burden of dealing with an advanced riemannian
geometric machinery.

Our paper is organized as follows: in \secref{sezvarapproach} we describe a rather general approach to reduction for biharmonic maps. In \secref{mappe-tra-modelli} we illustrate how these methods can be applied to the study of a family of equivariant biharmonic maps. Next, in \secref{G-invariant-immersions} we illustrate some applications of the reduction method to the study of large families of $G-$invariant submanifolds. Finally, in \secref{open-problems} we state some open problems and indicate some directions for possible developments.

\section{A reduction method for biharmonic maps}\label{sezvarapproach}

In order to introduce the variational context in which we work, let us first recall some basic facts concerning equivariant maps and their associated
(1-dimensional) variational problems. Various, rather general approaches are possible (see, for example, \cite{PBJCW, ER,Xin}): here we adopt the setting of \cite{Xin}, as developped in \cite{MR}. In particular, we shall consider maps $\varphi$ which are \emph{equivariant with respect to Riemannian submersions}. That amounts to require that the following be a commutative diagram:
\begin{equation}\label{diagrammacommutativo}
\begin{CD}
M @>\varphi>>N\\
@V \pi_1VV @V V\pi_2V\\
\bar{M} @>\bar{\varphi}>> \bar{N}
\end{CD}
\end{equation}
where $\pi_1 : M \to \bar{M}$ and $\pi_2 : N \to \bar{N}$ are Riemannian submersions.
We say that a vector field ${\mathcal V}$ along $\varphi$ is \emph{basic} if: for any $p\in M$, ${\mathcal V}(p)$ is horizontal with respect to $\pi_2$ and there exists a vector field $\bar{{\mathcal V}}$ along $\bar{\varphi}$ such that
\begin{equation}\label{basicvectorfield}
d\pi_2({\mathcal V}(p))=\bar{{\mathcal V}}(\pi_1(p)) \,\, .
\end{equation}
The following lemma was obtained by Xin (see Theorem 6.5 of \cite{Xin}) for tension field, harmonic maps and energy. The adaptation to the biharmonic case was proved in \cite{MR}:
\begin{lemma}\label{adaptationofXin}
Let $\varphi:M \to N $ be an equivariant map with respect to Riemannian submersions.  If its tension field
$\tau$ is a basic vector field, then $\varphi$ is harmonic if and only if it is a critical point of the energy with respect to all equivariant variations.

If its bitension field
$\tau_2$ is a basic vector field, then $\varphi$ is biharmonic if and only if it is a critical point of the bienergy with respect to all equivariant variations.
\end{lemma}

In order to apply this setting to the construction of harmonic and biharmonic maps, we first restrict our attention to the case that both the base manifolds in \eqref{diagrammacommutativo} are 1-dimensional. More specifically, we consider equivariant maps $\varphi_{\alpha}:(M,g)\to(N,h)$ and denote by
\begin{equation}\label{alfa}
\alpha :I \to J,\qquad  I,J\subseteq \R\quad \text{open intervals},
\end{equation}
the associated map between the base manifolds. Note that this function $\alpha$ may have to satisfy suitable boundary conditions dictated by the geometry of the problem.  If these maps have enough symmetries (in particular, to ensure that the tension field is basic) (see \cite{ER} for details), then the energy functional \eqref{energia} takes the following form:
\begin{equation}\label{energiaridotta}
    E(\varphi_\alpha)= \int_a^b \,\, L(t, \alpha, \dot{\alpha}) \, dt \,\, ,
\end{equation}
where $[a,b]$ is an arbitrary compact interval of $I$, $\dot{\alpha}$ denotes derivative with respect to $t$, and $L(\cdot, \cdot, \cdot)$ is a differentiable function depending on the geometry of the problem under consideration. Now, according to Lemma~\ref{adaptationofXin} (see \cite{Xin}), the harmonicity of the map $\varphi_{\alpha}$ is equivalent to the fact that $\alpha $ is a {\it critical point} of the so-called reduced energy functional \eqref{energiaridotta}. That is to say, the function $\alpha$ must be a solution on $I$ of the Euler-Lagrange equation associated with \eqref{energiaridotta}, which is the following:
\begin{equation}\label{eqeulero1}
    \frac{\partial L}{\partial \alpha} - \frac{d}{dt} \,  \frac{\partial L}{\partial \dot{\alpha}}  \, = \, 0 \,\, .
\end{equation}
For future reference, it is useful to recall how equation \eqref{eqeulero1} is derived from \eqref{energiaridotta}. The first step in this direction is to recognize that the requirement that $\alpha$ be critical is equivalent to the vanishing of the directional derivative
\begin{equation}\label{derivdir}
    \frac{d}{dh} \, \left [  E(\,\varphi_{(\alpha \, +\, h\, \beta) }\,)\,\right ] |_{h=0}
    \end{equation}
for all $\beta:I\to\R$ which vanish outside $(a,b)$, where $h\in(-\varepsilon,\varepsilon)$ for a suitable $\varepsilon>0$. Next, we compute \eqref{derivdir} explicitly:
\begin{eqnarray} \label{calcolo1}
\nonumber
  \frac{d}{dh} \, \left [  E(\,\varphi_{(\alpha \, +\, h\, \beta) }\,)\, \right ] |_{h=0} &=& \frac{d}{dh}\left ( \int_a^b \,\, L(t, \alpha+h\beta, \dot{\alpha}+h\dot{\beta})\, dt  \right )\Big {|}_{h=0} \\ \nonumber
   &=& \int_a^b \,\left (\frac{d}{dh}\, L(t, \alpha+h\beta, \dot{\alpha}+h\dot{\beta})|_{h=0} \right )\, dt  \\ \nonumber
   &=& \int_a^b \,\left (  \frac{\partial L}{\partial \alpha} \, \beta +  \frac{\partial L}{\partial \dot{ \alpha}} \, \dot{\beta}\right ) \, dt \\
   &=& \int_a^b \left [ \left ( \frac{\partial L}{\partial \alpha} - \frac{d}{dt} \,  \frac{\partial L}{\partial \dot{\alpha}} \right ) \, \beta \right ] \, dt \,\, ,
\end{eqnarray}
where, in order to obtain the fourth equality of \eqref{calcolo1}, we have used the fact that, since $\beta$ vanishes in $a$ and $b$:
\begin{equation}\label{supportocompatto}
0= \int_a^b \,\frac{d}{dt}\left (\frac{\partial L}{\partial \dot{\alpha}} \, \beta \right ) \, dt = \int_a^b \,\left (\frac{d}{dt}\, \frac{\partial L}{\partial \dot{\alpha}}  \right ) \, \beta\, \, dt + \int_a^b \, \frac{\partial L}{\partial \dot{\alpha}} \,\, \dot{\beta} \, \,dt \,\,.
\end{equation}
By way of summary, we conclude from \eqref{calcolo1} that the vanishing of the directional derivative \eqref{derivdir}, for all $[a,b]$ and for all variations $\{\alpha \, +\, h\, \beta\}_h$ of $\alpha$ supported in $[a,b]$, is equivalent to the Euler-Lagrange equation \eqref{eqeulero1}.
\\

Now, it is useful to illustrate in detail an important example.
\begin{example}\label{esempio-mappe-rotaionally-simmetric-models}
We describe equivariant maps between \emph{models} in the sense of \cite{GW} (see also \cite{Petersen}). More precisely, an $m$-dimensional manifold $(M^m(o),\,g)$ with a pole $o$ is a model if and only if every linear isometry of $T_oM$ can be realized as the differential at $o$ of an isometry of $M$. A significant geometric property of a model is the fact that we can describe it by means of geodesic polar coordinates centered at the pole $o$, as follows:
\begin{equation}\label{model}
(M^m(o),\,g) = \left (\,\s^{m-1} \times [0,\, + \infty) , \, f^2(r)\, g_{\s^{m-1}} \, + \, dr^2 \,\right ) \,\, ,
\end{equation}
where
$$ (\,\s^{m-1},\, g_{\s^{m-1}} \, )
$$
denotes the $(m-1)$-dimensional Euclidean unit sphere, and the function $f(r)$ is a smooth function which satisfies
\begin{equation}\label{condizioni-su-f}
\left \{
  \begin{array}{l}
    f(0)=0 \,, \quad f'(0)=1 \quad {\rm and}\quad f(r)>0 \quad {\rm if} \,\, r>0 \,\, ; \\
    \,\\
    f^{(2k)}(0)=0 \quad {\rm for} \,\, {\rm all } \,\, k \geq 1 \,\, .\\
  \end{array}
\right .
\end{equation}
We also note that $r$ measures the geodesic distance from the pole $o$. To shorten notation and emphasize the role of the function $f$, we shall write $M_f^m(o)$ to denote a model as in \eqref{model}.

\begin{remark}\label{remarksuR^meH^m} We observe that, if $f(r)=r$, then $M_f^m(o)=\R^m$. We also point out that, if $f(r)= (1\slash c) \, \sinh (r \,c)$ ($c>0$), then $M_f^m(o)$ represents $H^m(-c^2)$, i.e., the $m$-dimensional hyperbolic space of constant sectional curvature $-\, c^2$. With a slight abuse of terminology, in our study of equivariant biharmonic maps we shall also consider the case that $f(r)$ is defined on a finite interval $[0,\,b]$, with $f(b)=0$, $f'(b)=-1$ and  $f^{(2k)}(b)=0$  {for} all  $k \geq 1$. In particular, we shall pay special attention to the case $f(r)= (1\slash~d)\, \sin (r \, d)$, where $d>0$ and $\, 0 \leq r \leq (\pi \slash d )\,$: in this case, our manifold is the Euclidean $m$-sphere $\s^m(d^2)$ of constant sectional curvature $d^2$.
\end{remark}

For future use, we also recall that the radial curvature $K(r)$ ($r>0$) of a model $M_f^m(o)$ is defined as the sectional curvature of any plane which contains $\partial \slash\, \partial  r$. The radial curvature is related to the function $f(r)$ by means of the following fundamental equation (the Jacobi equation, see \cite{GW}):
\begin{equation}\label{equazionecurvaturaradiale}
    f''(r)\,+\,K(r)\,f(r)\,=\,0 \,\, , \quad r>0 \,\, .
\end{equation}

Our reduction technique applies, in particular, to the study of \emph{rotationally symmetric} maps between models, i.e.:
\begin{eqnarray}\label{equivariantsymmetricmaps}
\nonumber
    \varphi_{\alpha}\,: \,\,M_f^m(o) &\to&
    \,\,M' _h\,^m (o') \\
    ( \theta, \, r) \, &\mapsto& \, (\theta, \, \alpha(r)) \,\, ,
\end{eqnarray}
where the function $\alpha(r)$ is smooth and positive on $(0,\,+\infty)$  and satisfies specific boundary conditions. For example, for $ \varphi_{\alpha}:\R^m\to\s^m$, we assume the following boundary condition for $\alpha$ which ensures the continuity of $\varphi_{\alpha}$ across the pole:
\begin{equation}\label{boundaryconditionforalfa}
                  \alpha(0)=0 \,\, .
\end{equation}

\begin{remark} More generally, \emph{equivariant} maps of the following type can be studied by similar methods:
\begin{eqnarray}\label{equivariantsymmetricmaps-non-usata}
\nonumber
    \varphi_{\alpha}\,: \,\,M_f^m(o) &\to&
    \,\,M' _h\,^n (o') \\
    ( \theta, \, r) \, &\mapsto& \, (\Psi_\lambda(\theta), \, \alpha(r)) \,\, ,
\end{eqnarray}
where $\Psi_\lambda(\theta)$ is a so-called {\it eigenmap} of eigenvalue $\lambda$ . That means that
$\Psi_\lambda:\, \s^{(m-1)}\to \s^{(n-1)}$ is a harmonic map with {\it constant} energy density equal to $(\lambda \slash 2)$. Important examples of eigenmaps are: the identity map of $\s^{(m-1)}$ ($\lambda=m-1$), the k-fold rotation $e^{i\theta} \rightsquigarrow e^{ik\theta}$ of $\s^1$ ($\lambda=k^2$); and, also, the Hopf fibrations $\s^3 \to \s^2$, $\s^7 \to \s^4$ and $\s^{15} \to \s^8$, with $\lambda$ equal to 8, 16 and 32 respectively. In the special case that $\Psi_\lambda$ is the identity map we recover the above case of rotationally symmetric maps.
\end{remark}

In this paper we shall focus on rotationally symmetric maps: in this case, the Lagrangian $L$ corresponding to the reduced energy \eqref{energiaridotta} is given (writing $r$ instead of $t$ for geometric convenience) by:
\begin{equation}\label{equivariant-energy}
   L(r, \alpha, \dot{\alpha})\,=\, \frac{1}{2}\,\, \left [\,\dot{\alpha}^2 + \, (m-1) \, \frac{h^2(\alpha)}{f^2(r)} \, \right ]\, V(r)\,\,,
\end{equation}
where the volume function $V(r)$ in this case is (up to an irrelevant multiplicative positive constant)
\begin{equation}\label{volume-per-mappe-modelli}
   V(r)= f^{m-1}\,(r)\,\,.
\end{equation}
\end{example}


\begin{remark}\label{remark-relazione-tension-volume} In the notation of \eqref{decrescenzarapidaenergia}, we observe that
\begin{equation}\label{significato-eq-eulero-lagrange-energia}
     \frac{d}{dh} \, \left [  E(\,\varphi_{(\alpha \, +\, h\, \beta) }\,)\, \right ] |_{h=0} \,\, =\,\,  \nabla_{\mathcal V} \, E(\varphi_{\alpha}) \,\, ,
\end{equation}
where ${\mathcal V}$ is the vector field along $\varphi_{\alpha}$ given by
\begin{equation}\label{definizione-di-v}
    {\mathcal V}= \beta \, \frac{\partial}{\partial \alpha}
\end{equation}
(here and below we identify horizontal vector fields with respect to $\pi_2$ with their projection on $\bar{N}=\R$). Now, comparing the result of the computations \eqref{calcolo1} with \eqref{decrescenzarapidaenergia}, we deduce that (again, up to an irrelevant multiplicative positive constant)
\begin{equation}\label{relazione-tra-tau-e-eq-eulero}
    \tau (\varphi_{\alpha})=- \frac{A_{\alpha}\,\,(r, \alpha, \dot{\alpha})\,}{V(r)} \,\,\, \frac{\partial}{\partial \, \alpha} \,\, ,
\end{equation}
where we have set:
\begin{equation}\label{definizione-di-A-alfa}
    A_{\alpha}\,(r, \alpha, \dot{\alpha})\,\,=\, \,\frac{\partial L}{\partial \alpha}- \frac{d}{dr} \,  \frac{\partial L}{\partial \dot{\alpha}}\,.
\end{equation}

\end{remark}

We are now in the right position to extend this setting to the framework of biharmonic maps. In the equivariant context, the bienergy functional \eqref{bienergia} takes the following form:
\begin{equation}\label{bienergiaridotta}
    E_2 (\varphi_\alpha)= \int_a^b \,\, L(t, \alpha, \dot{\alpha}, \ddot{\alpha}) \, dt \,\, ,
\end{equation}
where, as above, the function $\alpha :I \to J $ may have to satisfy suitable boundary conditions, $L(\cdot, \cdot, \cdot, \cdot)$ is a differentiable function depending on the geometry of the problem under consideration and $[a,b]$ is an arbitrary interval of $I$. Now, according to Lemma~\ref{adaptationofXin}, the condition of biharmonicity for an equivariant map $\varphi_{\alpha}$ with basic bitension field is equivalent to $\alpha $ being a critical point of the reduced bienergy functional \eqref{bienergiaridotta}. Next, following \cite{MR}, we state the most useful result in this context:

\begin{lemma}\label{reductiontheorem}
A differentiable function $\alpha : I \to J $ is a critical point of the reduced bienergy functional \eqref{bienergiaridotta} if and only if it is a solution of the following differential equation:
\begin{equation}\label{eqeulero2}
     \frac{\partial L}{\partial \alpha} - \frac{d}{dt} \,  \frac{\partial L}{\partial \dot{\alpha}}  + \frac{d\,^2}{dt^2} \, \frac{\partial L}{\partial \ddot{\alpha}}   \, = \, 0 \,\, .
\end{equation}
\end{lemma}

\begin{proof} Here the proof will only be outlined. Essentially, it amounts to compute the following directional derivative:

\begin{eqnarray} \label{calcolo2}
\nonumber  \frac{d}{dh} \, \left [  E_2\,\left(\varphi_{\alpha \, +\, h\, \beta}\,\right) \right ] |_{h=0} &=& \frac{d}{dh}\left ( \int_a^b \,\, L(t, \alpha+h\beta, \dot{\alpha}+h\dot{\beta}, \ddot{\alpha}+h\ddot{\beta})\, dt  \right )\Big {|}_{h=0} \\ \nonumber
   &=& \int_a^b \,\left (\frac{d}{dh}\, L(t, \alpha+h\beta, \dot{\alpha}+h\dot{\beta},\ddot{\alpha}+h\ddot{\beta})|_{h=0} \right )\, dt  \\ \nonumber
   &=& \int_a^b \,\left (  \frac{\partial L}{\partial \alpha} \, \beta +  \frac{\partial L}{\partial \dot{ \alpha}} \, \dot{\beta} + \frac{\partial L}{\partial \ddot{ \alpha}} \, \ddot{\beta}\right ) \, dt \\
   &=& \int_a^b \left [ \left ( \frac{\partial L}{\partial \alpha} - \frac{d}{dt} \,  \frac{\partial L}{\partial \dot{\alpha}} + \frac{d\,^2}{dt^2} \,  \frac{\partial L}{\partial \ddot{\alpha}}\right ) \, \beta \right ] \, dt \,\, .
\end{eqnarray}
Now, the vanishing of this directional derivative for all $[a,b]$ and all $\beta$ supported in $[a,b]$ is equivalent to the fact $\alpha$ is a solution of \eqref{eqeulero2}, as required.
\end{proof}

By way of summary, from Lemmata \eqref{adaptationofXin} and \eqref{reductiontheorem} we deduce (see \cite{MR}):

\begin{theorem}\label{bi-reduction-theorem} An \emph{equivariant} map $\varphi_{\alpha}$ with basic bitension field is \emph{biharmonic} if and only if $\alpha$ is a solution of \eqref{eqeulero2}.
\end{theorem}

\begin{remark}\label{differentiability-across-pole}
If one is able to produce a rotationally symmetric map $\varphi_{\alpha}:\R^m\to\s^m$ such that $\alpha(r)$ is a solution of \eqref{eqeulero2} which satisfies \eqref{boundaryconditionforalfa}, then one has to check the differentiability of $\varphi_{\alpha}$ across the pole. To this purpose, it is enough to verify that the following requirements are satisfied:
\begin{equation}\label{boundaryconditionforalfa-smoothness}
   \left \{
              \begin{array}{l}
                \alpha^{(2k)}(0)=0 \,\,\,\, \forall \, k \geq 1; \\
                \,\\
                \alpha^{(2k+1)}(0)<\infty \,\,\,\, \forall \, k \geq 0.\\
              \end{array}
            \right .
\end{equation}
In particular, it follows that $\alpha:[0,\infty)\to[0,\pi)$ is smooth.
\end{remark}

\begin{remark}\label{remark-equivariance-per-immersions}
In some important geometric applications, such as the study of certain $G-$invariant immersions (see \secref{G-invariant-immersions} below), the base manifold $\bar{N}$ in \eqref{diagrammacommutativo} is an $r$-dimensional orbit space, so that the base map $\overline{\varphi}$ is actually a curve and it is necessary to replace \eqref{bienergiaridotta} by:
\begin{equation}\label{bienergiaridottavettoriale}
    E_2(\varphi_{\alpha_j})= \int_a^b \,\, L(t, \alpha_j, \dot{\alpha_j}, \ddot{\alpha_j}) \, dt \,\, , \,\quad j= 1,\ldots,\,r \,\, .
\end{equation}
In this context, an argument similar to Lemma~\ref{reductiontheorem} leads us to the conclusion that the critical points of the functional \eqref{bienergiaridottavettoriale} are precisely the solutions of the following \emph{system} of ordinary differential equations:
\begin{equation}\label{eqeulero2vettoriale}
     \frac{\partial L}{\partial \alpha_j} - \frac{d}{dt} \,  \frac{\partial L}{\partial \dot{\alpha_j}}  + \frac{d\,^2}{dt^2} \, \frac{\partial L}{\partial \ddot{\alpha_j}}   \, = \, 0 \,\, , \,\quad j= 1,\ldots,\,r  \,\, .
\end{equation}
We shall clarify this case in \secref{G-invariant-immersions} below.
\end{remark}

\begin{remark} Our approach require explicitly that the equivariant map $\varphi_{\alpha}$ have basic bitension field. Therefore, it would be very interesting to obtain a complete geometric characterization of the situations in which this property is satisfied. In general, this appears to be a rather technical and difficult problem. However, here we just point out that there are important, large families of equivariant maps for which it is immediate to conclude that both the tension and the bitension fields are basic. In particular, that occurs when the Riemannian submersions $\pi_i \,, \,\, i=1,2$ in \eqref{diagrammacommutativo} are determined by isometric actions of Lie groups, and the equivariant maps $\varphi_{\alpha}$, when restricted to the fibres endowed with the induced metric, are harmonic maps with constant energy density. All the examples in \secref{mappe-tra-modelli} and \secref{G-invariant-immersions} below are of this type.
\end{remark}

\section{Rotationally symmetric biharmonic maps between models}\label{mappe-tra-modelli}

In this section we study rotationally symmetric maps between models, as defined in \eqref{equivariantsymmetricmaps}. Further details on this topic and related examples can be found in \cite{Montaldobis}, \cite{HM}, \cite{MOR-JJAA}, \cite{MR} and\cite{WangOuYang}.
First, by using \eqref{eqeulero1}, \eqref{equivariant-energy} and \eqref{relazione-tra-tau-e-eq-eulero}, we easily compute the tension field of a rotationally symmetric map as in \eqref{equivariantsymmetricmaps}. We obtain
\begin{equation}\label{tensionfieldtrawarped}
    \tau (\varphi_{\alpha})= \left [\ddot{\alpha}(r) + (m-1)\, \frac{f'(r)}{f(r)}\, \dot{\alpha}(r)- (m-1)\, \, \frac{h(\alpha)\,h'(\alpha)}{f^2(r)} \right ] \,\, \frac{\partial}{\partial \alpha} \,\,,\quad\text{on}\;\; (0,\infty).
\end{equation}
Therefore, in this case the reduced bienergy is given, up to an irrelevant multiplicative  constant factor, by the following expression:
\begin{equation}\label{reducedbienergiatrawarped}
    E_2(\varphi_\alpha) =  \frac{1}{2}\, \int_0^{+\infty } \left [\ddot{\alpha}(r) + (m-1)\frac{f'(r)}{f(r)}\, \dot{\alpha}(r)- (m-1)\frac{h(\alpha)\,h'(\alpha)}{f^2(r)} \right ]^2 f^{m-1}(r)\, dr\,\,.
\end{equation}
Next, we apply \eqref{eqeulero2} to \eqref{reducedbienergiatrawarped} and, after a long but straightforward computation we find that the condition of biharmonicity for rotationally symmetric maps of the type \eqref{equivariantsymmetricmaps} is the following (to simplify notation, we write $f$ and $h(\alpha)$ instead of $f(r)$ and $h(\alpha(r))$ respectively):
\begin{equation}\label{eqeulero2-esplicita}
\begin{aligned}
   & f^{m-5} \Big ((m-1) h(\alpha ) (2 f f''
   h'(\alpha )-2 (m-3) f f' \dot{\alpha} h''(\alpha
   )+2 (m-4) f'^2 h'(\alpha )\\
   &-f^2
   (h^{(3)}(\alpha ) \dot{\alpha}^2+2\ddot{\alpha}
   h''(\alpha ))+(m-1) h'(\alpha )^3)+f
   ((m-3) (m-1) f f'^2 \ddot{\alpha}\\
   &-(m-3) (m-1)
   f'^3 \dot{\alpha}+(m-1) f (f (f^{(3)}
   \dot{\alpha}+2 f'' \ddot{\alpha})-2 \ddot{\alpha}
   h'(\alpha )^2\\
   &-3 \dot{\alpha}^2 h'(\alpha ) h''(\alpha
   ))+(m-1) f' (\dot{\alpha} ((m-4) f
   f''-2 (m-3) h'(\alpha )^2)+2 f^2 \alpha
   ^{(3)})\\
   &+f^3 \alpha ^{(4)})+(m-1)^2
   h(\alpha )^2 h'(\alpha ) h''(\alpha )\Big )=0\,\, .\\
\end{aligned}
\end{equation}

\begin{remark}
We point out that, to  the purpose of comparison  with the equation
given for $m=2$ in \cite[Corollary~2.3]{WangOuYang}, equation \eqref{eqeulero2} can be rewritten as follows:
$$
\begin{cases}
F''+(m-1) \dfrac{f f' F'- h'(\alpha)^2 F}{f^2}-(m-1) \dfrac{h(\alpha) h''(\alpha)F}{f^2}=0\\
F=\ddot{\alpha}+(m-1) \dfrac{f'}{f} \dot{\alpha} -(m-1) \dfrac{h(\alpha) h'(\alpha)}{f^2}\,\,.
\end{cases}
$$
\end{remark}
In summary, a rotationally symmetric map $\varphi_{\alpha}$ as in \eqref{equivariantsymmetricmaps} is \emph{biharmonic} if and only if $\alpha$ is a solution of \eqref{eqeulero2-esplicita}.

At this stage, it is worth to point out two facts. Our computation of the condition of biharmonicity has avoided the explicit use (see equations \eqref{secondaformalocale} and \eqref{bitensionfield}) of the Christoffel symbols and their related riemannian geometric machinery. On the other hand, just a short look at \eqref{eqeulero2-esplicita} reveals the complexity of the analytical task required to deal with a fourth order differential equation. \\


The difficulty of the general problem suggests to restrict investigation to speficic, geometrically significant, families of maps. In particular, we now look for  rotationally symmetric, proper biharmonic \emph{conformal} diffeomorphisms between $m$-dimensional models of constant sectional curvature. We shall be able to obtain a complete description of such maps.
First, let us point out when a rotationally symmetric map as in \eqref{equivariantsymmetricmaps} is conformal. Comparing dilations of vectors which are
respectively orthogonal and tangent to the radial direction, we easily find that a map of the type \eqref{equivariantsymmetricmaps} is conformal iff
$$
    \dot{\alpha}^2=\frac{h^2(\alpha)}{f^2(r)} \,\,
    $$
or, taking into account the boundary condition \eqref{boundaryconditionforalfa},
\begin{equation}\label{equazionediconformalita}
    \dot{\alpha}=\frac{h(\alpha)}{f(r)} \,\, .
\end{equation}
Using \eqref{equazionediconformalita} into the biharmonicity equation \eqref{eqeulero2-esplicita}, we find that the condition of biharmonicity for a conformal map of the type \eqref{equivariantsymmetricmaps} is the following:

\begin{equation}\label{equazionebiarmonicaeconformeperognim}
   \begin{aligned}
& (m-2) f^{m-5} h(\alpha )\Big(f^2 f^{(3)}+h'(\alpha
   ) \big(4 f f''+(m-5) h(\alpha ) h''(\alpha
   )\big)\\
   &+(3 m-14) f'^2 h'(\alpha )-2 (m-4) f'^3\\
   &+f'
   \big((m-7) f f''-2 (m-4) h(\alpha ) h''(\alpha )-2
   (m-4) h'(\alpha )^2\big)\\
   &- h(\alpha )^2 h^{(3)}(\alpha
   )+(m-2) h'(\alpha )^3\Big) = 0\,\, .
\end{aligned}
\end{equation}
In particular, if $m=4$, equation \eqref{equazionebiarmonicaeconformeperognim} becomes:
\begin{equation}\label{equazionebiarmonicaeconformem=4}
\begin{aligned}
&\frac{2 h(\alpha)}{f} \,\,  \Big(f^2 f^{(3)}+h'(\alpha)
   \big(4 f f''-h(\alpha) h''(\alpha)\big) \\
   & -2
   f'^2 h'(\alpha )-3 f f' f''-h(\alpha )^2
   h^{(3)}(\alpha)+2 h'(\alpha)^3\Big) =0 \,\, .
   \end{aligned}
\end{equation}

Now, we are in the position to provide a complete description of solutions in the case of maps between 4-dimensional models of constant sectional curvature. More precisely, we have the following result of \cite{MOR-JJAA}:

\begin{theorem}\label{proposizionediclassificazione} Let us consider rotationally symmetric maps as in \eqref{equivariantsymmetricmaps}-\eqref{boundaryconditionforalfa}. Let us assume that $m=4$ and that both models have constant sectional curvature $0,1$ or $-1$ (see Remarks~\ref{remarksuR^meH^m}). Then the  biharmonic conformal diffeomorphisms of type  \eqref{equivariantsymmetricmaps}-\eqref{boundaryconditionforalfa} can be enumerated as follows ($c$ denotes a real positive constant):
\begin{enumerate}
\item[\textbf{Case 1}] - $f(r)=r\,.$
\begin{enumerate}
\item[A] - $h(\alpha)=\alpha$, $\alpha(r)= c\,r$ (harmonic diffeomorphisms from $\R^4$ to itself);
\item[B]- $h(\alpha)= \sin(\alpha)$ and
\begin{equation}\label{definizionesoluzionedaeuclideoasfera}
    \alpha(r)= \,\,2\, \arctan (c^2\,r)\,\,\,\, (r\geq 0)
\end{equation}
(proper biharmonic diffeomorphisms from $\R^4$ to $\s^4\smallsetminus\{{\rm south \,\, pole}\}$);
\item[C]- $h(\alpha)=\sinh( \alpha)$ and
\begin{equation}\label{definizionesoluzionedaeuclideoaiperbolico}
    \alpha(r)= \,\, \,\,2 \, \tanh ^{-1} (c^2\,r)\,\,\,\, (0 \leq r < \frac{1}{c^2})
\end{equation}
(proper biharmonic diffeomorphisms from $B^4(1 \slash c^2)$ (i.e., the open ball of radius $(1 \slash c^2)$ in $\R^4$) to $\H^4$).
\end{enumerate}
\item[\textbf{Case 2}] -  $f(r)=\sin(r).$
\begin{enumerate}
\item[A] - if $h(\alpha)=\alpha$, then there is no solution;
\item[B]-  if $h(\alpha)=\sin(\alpha)$, then
$$
\alpha (r)= r\,\,(0 \leq r < \pi)
$$
gives rise to a harmonic conformal diffeomorphism, but in this case we do not have proper biharmonic examples;
\item[C]-  if $h(\alpha)= \sinh(\alpha)$, then there is no solution.
\end{enumerate}
\item[\textbf{Case 3}] -  $f(r)=\sinh(r).$
\begin{enumerate}
\item[A] - if $h(\alpha)=\alpha$, then there is no solution;
\item[B]-  if $h(\alpha)=\sin(\alpha)$, then there is no solution;
\item[C]-  if $h(\alpha)=\sinh(\alpha)$, then
$$
\alpha (r)= \,r \,\, (r \geq 0),
$$
 produces a harmonic conformal diffeomorphism, while there is no proper biharmonic example.
 \end{enumerate}
 \end{enumerate}
\end{theorem}
\begin{proof}\label{proofofproposizionediclassificazione} The proof amounts to a case by case analysis of \eqref{equazionebiarmonicaeconformem=4} and we omit it.
\end{proof}

\begin{remark} Direct inspection shows that all the solutions given in \eqref{definizionesoluzionedaeuclideoasfera} and \eqref{definizionesoluzionedaeuclideoaiperbolico} satisfy \eqref{boundaryconditionforalfa-smoothness} and so they extend smoothly across the pole.
\end{remark}

\begin{remark}
We point out that \eqref{definizionesoluzionedaeuclideoasfera} and \eqref{definizionesoluzionedaeuclideoaiperbolico} correspond to a dilation of $\R^4$ composed with the inverse of the stereographic projection of $\s^4$ and $\H^4$ respectively. The map $\varphi_{\alpha}:\R^4\to\s^4$ associated to the solution \eqref{definizionesoluzionedaeuclideoasfera} also provides an example of {\it extrinsic} biharmonic map (see \cite{Cooper}). 
\end{remark}

\begin{remark} A direct, case by case inspection shows that, if $m>2$ and $m\neq4$, then there exists no rotationally symmetric, conformal, proper biharmonic map between $m$-dimensional models of constant sectional curvature. By way of example, if $f(r)=r$ and $h(\alpha)= \sin \alpha$, then condition \eqref{equazionebiarmonicaeconformeperognim} becomes:
\begin{equation}\label{casogeneraleconformebiharmonicesempio}
    4\,(m-2)\,(m-4)\,r^{m-5}\,\sin (2\,\alpha) \,\sin^4 (\alpha\slash2) \,=\,0 \, \, ,
\end{equation}
which is not possible when $m \neq 2,\,4$. The other cases are similar, so we omit further details.
\end{remark}

\begin{remark}\label{Jager-Kaul} It is interesting to compare the proper biharmonic diffeomorphisms from $\R^4$ to $\s^4 \smallsetminus \{{\rm south \,\, pole} \}$ with the qualitative behaviour of rotationally symmetric harmonic maps (see \cite{JK}). In particular, J\"{a}ger and Kaul proved that the image of the functions $\alpha(r)$ associated to these harmonic maps cover a range $[0,\,R_4]$, with $(\pi \slash 2) < R_4 < \pi$, and $\alpha(r)$ oscillates around $(\pi \slash 2)$ as $r \rightarrow +\infty$ (further details concerning the numerical value of $R_4$ can be found in \cite{JK}).
J\"{a}ger and Kaul applied their results to draw some interesting conclusions concerning the existence of rotationally symmetric solutions to the Dirichlet problem for maps from the Euclidean unit $m$-ball $B^m$ to $S^m$. In particular, if $m=4$, they proved that the Dirichlet problem with boundary data
\begin{equation}\label{dirichletboundarydata}
    (\theta,\,1) \rightarrow (\theta,\,\alpha(1)=R^*)
\end{equation}
admits a rotationally symmetric solution if and only if:
\begin{equation}\label{soluzionijagerkaul}
    0 \leq R^* \leq R_4 \,\,.
\end{equation}
By contrast, the existence of the  proper biharmonic conformal diffeomorphisms of Theorem~\ref{proposizionediclassificazione} implies that the boundary value problem \eqref{dirichletboundarydata} admits proper biharmonic solutions for all
$$
0 \leq R^* < \pi \,\,.
$$
\end{remark}

\begin{remark}\label{lavoroLoubeaueBaird} The explicit solutions of Cases 1B and 1C above were also studied in \cite{Baird-et-al} and \cite{Loubeau-et-al}. In particular, it was shown in \cite{Loubeau-et-al} that they do not provide examples of biharmonic morphisms.
\end{remark}

A natural, general development of Theorem~\ref{proposizionediclassificazione} is to study when a \emph{conformal}, rotationally symmetric map as in \eqref{equivariantsymmetricmaps} is proper biharmonic.

Since, once again, the general case appears to be difficult, we analyze the case that the domain model is the Euclidean space, obtaining a complete answer when the domain model is $\R^4$.
Our first result is the following proposition.

\begin{proposition}
Let $M' _h\,^4 (o')$ be a $4-$dimensional model and assume that the function $h$ satisfies the additional condition
\begin{equation}\label{biharmonic+conformal}
    \, h^2\, h''' \,+h'\,(2+h\,h'')\,-2\,h'^3=\,0 \,\, .
\end{equation}
Then any rotationally symmetric conformal diffeomorphism $\varphi_{\alpha}:  \R^4 \to
   {M' _h}^4 (o')$ is biharmonic.
\end{proposition}
Condition \eqref{biharmonic+conformal} is obtained  by \eqref{equazionebiarmonicaeconformem=4} using $f(r)=r$. In particular, it must be verified whenever 
$\varphi_{\alpha}:  \R^4 \to {M' _h}^4 (o')$ is a rotationally symmetric biharmonic conformal diffeomorphism. Of course, $h(\alpha)=\alpha$, $h(\alpha)=\sin(\alpha)$ and $h(\alpha)=\sinh(\alpha)$ are solutions of \eqref{biharmonic+conformal}. The next result basically states that the examples provided in Theorem~\ref{proposizionediclassificazione} are the only rotationally symmetric proper biharmonic conformal diffeomorphisms satisfying the boundary condition \eqref{boundaryconditionforalfa}.

\begin{theorem}\label{maintheorem} 
Let $M' _h\,^4 (o')$ be a $4-$dimensional model and assume that the function $h$ satisfies
\eqref{biharmonic+conformal}. Then $M' _h\,^4 (o')$ has constant sectional curvature.
\end{theorem}
\begin{proof}
First, we observe that
\begin{equation}\label{hamiltonianasuerre4}
   h^2\, (1-h'^2+ h\,h'')=c \,\, ,
\end{equation}
where $c$ is a real constant, is a prime integral of \eqref{biharmonic+conformal}. As the function $h$ satisfies the boundary condition \eqref{boundaryconditionforalfa}, the constant $c$ must be zero and therefore 
%
\begin{equation}\label{equazionechiave}
    1-h'^2+ h\,h''=0 \,\, .
\end{equation}
Next, derivation of \eqref{equazionechiave} leads us to conclude that
\begin{equation}\label{equazionechiavebis}
     h\,h'''-h' \, h''=0 \,\, .
\end{equation}
Now, using the Jacobi equation \eqref{equazionecurvaturaradiale} into \eqref{equazionechiavebis}, we deduce that, 
\begin{equation}\label{equazionechiavetris}
     h'''= -\,K\, h' \,\, .
\end{equation}
On the other hand, taking derivatives on both sides of the Jacobi equation \eqref{equazionecurvaturaradiale}, we obtain
\begin{equation}\label{equazionechiavequadris}
     h'''= -\,K\, h' - K'\,h\,\, .
\end{equation}
Finally, comparing \eqref{equazionechiavetris} and \eqref{equazionechiavequadris}, we conclude that $K' \equiv 0$, from which it follows that  the radial curvature $K$ is constant, 
a fact which implies that the sectional curvature is constant  (see \cite{GW}).
\end{proof}


\section{$G-$invariant biharmonic immersions}\label{G-invariant-immersions}

We shall work in the framework of equivariant differential geometry, as introduced by means of the diagram \eqref{diagrammacommutativo}. Let $I(N)$ be the full isometry group of the riemannian manifold $N$: it is well-known (see \cite{MST}) that $I(N)$ is a Lie group which acts differentiably on $N$. A Lie subgroup $G$ of $I(N)$ is called an {\it isometry group} of $N$ and, following \cite{HL}, we recall that its {\it cohomogeneity} is defined as the codimension in $N$ of the maximal dimensional orbits, also called the {\it principal} orbits (of course, all the orbits are homogeneous spaces, since they are of the type $G \, / \,H$, where $H$ is the stabilizer). The \emph{cohomogeneity of a $G-$invariant submanifold} $M$ of $N$ is defined as the dimension of $M$ minus the dimension of the principal orbits.

We shall focus on the study of hypersurfaces in the case that $N$ is the Euclidean space $\R^n$ and the cohomogeneity of $G$ is two, so that \emph{$G-$invariant hypersurfaces are cohomogeneity one submanifolds}. This type of isometry groups of $\R^n$ have been fully classified in \cite{HL}. In particular, Hsiang and Lawson, developping the work of various famous authors, including Cartan and Weil, divided the cohomogeneity two isometry groups $G$ acting on $\R^n$ into five types according to the geometric shape of their orbit space $Q=\bar{N}=\R^n \, / \,G$, which is a linear cone in $\R^2$ of angle $\pi \slash d$, $d=1,\,2,\,3,\,4,\,6$ respectively. In this context, a $G-$invariant hypersurface in $\R^n$ can be completely described by means of its profile curve $\gamma$ into the $2$-dimensional orbit space $Q$. In particular, as a special instance of the reduction theory illustrated in \secref{sezvarapproach}, it turns out that a $G-$invariant hypersurface is biharmonic  if and only if the curve $\gamma$ satisfies a certain system of ordinary differential equations (see Proposition~\ref{equazioniridottedibiarmonicity} below for details). In general, as we pointed out in the introduction, reduction to an ODE has been a valuable tool because it has helped to produce new solutions, such as counterexamples for Bernstein's type problems for minimal and CMC submanifolds (see, for example, \cite{BDG,ER, Hsiang}). By contrast, in our case what we obtain (see \ref{Main-theorem} below) is a nonexistence result in the direction of the still open Chen conjecture (see \cite{Chen04,Chen} and \cite[Chapter~7]{Chen-book}): {\it biharmonic submanifolds into $\R^n$ are minimal}.

Chen's conjecture is still open even for biharmonic hypersurfaces in $\R^n$ though, by a result of Dimitric (see \cite{dimitric}), we know that any biharmonic hypersurface in $\R^n$ with at most two distinct principal curvatures is minimal. Other partial results for low dimensions state that biharmonic hypersurfaces with at most three distinct principal curvatures in $\R^4$ or in $\R^5$ are necessarily minimal (see \cite{Defever, HasVla95}) and very recently Fu  (see \cite{Fu2014}, \cite{Fu2014A}) extended Dimitric's result proving that any biharmonic hypersurface in $\R^n$ with at most three distinct principal curvatures is minimal.  We can now state the main nonexistence result in this context (see \cite{MOR-G-Invariant-biharm-immersions}):

\begin{theorem}\label{Main-theorem} Let $G$ be a cohomogeneity two group of isometries acting on $\R^n$ ($n\geq3$). Then any $G-$invariant biharmonic hypersurface in $R^n$ is minimal.
\end{theorem}

\begin{remark}The family of $G-$invariant hypersurfaces in Theorem~\ref{Main-theorem} is ample and geometrically significant (see Table~\ref{table} below). The case $d=1$ (i.e., the case of the classical rotational hypersurfaces) is a special instance of the above cited result of \cite{dimitric}. The next case, i.e., $d=2$ and $G=SO(p)\times SO(q)$, with $p+q=n$, was proved in \cite{MOR-AMPA}. For these reasons, the interest of Theorem~\ref{Main-theorem} lies on the three remaining types of $G-$actions, i.e., $d=3,\,4$ and $6$, for which the number of distinct principal curvatures is $4,\,5$ and $7$ respectively.
\end{remark}

In order to set Theorem~\ref{Main-theorem} in our reduction context it is convenient to begin by recalling some basic facts from equivariant differential geometry. The details, together with some historical references, concerning these results can be found in \cite{BackDoCarmoHsiang, HL, Pedrosa}. So, let $G$ be a cohomogeneity two group of isometries acting on $\R^n$. The following linear functions $w_{(d,i)}$ will play an important role:
\begin{equation}\label{definizionew_{(d,i)}}
    w_{(d,i)}(x,y)= x\, \sin (i\,\pi \slash \,d)\,-\, y\,\cos (i\,\pi \slash \,d)\,\, ,
\end{equation}
where $d$ is an integer which can be equal to $1,\,2,\,3,\,4$ or $6$, and $i$ is another integer such that $0 \leq i \leq (d-1)$. The orbit space $\bar{N}=Q=\R^n\, / \,G$ can be identified with a linear cone of angle $(\pi \slash \,d)$ in $\R^2$ described by
\begin{equation}\label{descrizionediQ}
    Q=\left \{\, (x,y) \,\in \, \R^2 \,\, : \,\, y \geq 0 \,\,{\rm and}\,\,x\, \sin (\pi \slash \,d)\,-\, y\,\cos (\pi \slash \,d)\geq 0  \,\right \} \,\,,
\end{equation}
where the possible cases for $d$ are $1,\,2,\,3,\,4$ or $6$. The orbital distance metric $g_Q$ on $Q$ (i.e., the metric which makes the projection map $\pi_2 \,\, : \R^n \rightarrow Q$ in the diagram \eqref{diagrammacommutativo} a Riemannian submersion) is flat:
\begin{equation}\label{metricadiQ}
    g_Q=dx^2+dy^2
\end{equation}
and any horizontal lift of a tangent vector to $Q$ meets any $G-$orbit perpendicularly.

Let $\xi=(x,y)$ be an interior point of $Q$. We denote by $V(\xi)$ the volume of the principal orbit $\pi_2^{\,-1}\,(\xi)$. The function $V(\xi)$ is called the \emph{volume function} and contains most of the information required to carry out the computation of the second fundamental form $A$ associated to a $G-$invariant hypersurface. More precisely, it turns out that $V(x,y)$ is always a homogeneous polynomial which, for each fixed type $d$, can be expressed (up to a multiplicative constant) in terms of the linear functions \eqref{definizionew_{(d,i)}} in the following form:
\begin{equation}\label{genericafunzionevolume}
    V^2(x,y)= \prod_{i=0}^{(d-1)}\,  \left [\,w_{(d,i)}(x,y)\,\right ]^{2\,m_i} \,\, ,
\end{equation}
where the $m_i$'s are positive integers (the cases which can occur are listed in Table~\ref{table}, which can be derived from an analogous table given in \cite{Hsiang,HL}).
\begin{table}[b]
  \centering
  \scriptsize{
\begin{tabular}{|c|c|c|c|c|}
\hline
&&&&\\
  $G$ & Action & $\dim$ Euclidean Space  & $d$ & Multiplicities \\
  &&&&\\
  \hline
  &&&& \\
  $S\mathcal{O}(n-1)$&$1+\rho_{(n-1)}$&$n\geq3$&1&$m_0=(n-2),\,m_1=1$\\
   &&&& \\
  \hline
  &&&& \\
  $S\mathcal{O}(p)\times S\mathcal{O}(q)$&$\rho_p+\rho_{q}$&$n=p+q, \,p,\,q \geq 2$&2&$m_0=(q-1),\,m_1=(p-1)$\\
  &&&& \\
  \hline
&&&&\\
  $S\mathcal{O}(3)$ &$S^2_{\rho_{3}}-1$ &  $n=5$ & 3 & $m_0=m_1=m_2=1$ \\
  &&&&\\
  \hline
  &&&&\\
  $SU(3)$ &$Ad$ &  $n=8$ & 3 & $m_0=m_1=m_2=2$ \\
  &&&&\\
  \hline
  &&&&\\
  $Sp(3)$ & $\Lambda^2 \, \nu_{3}-1$& $ n=14$ & 3 & $m_0=m_1=m_2=4$ \\
  &&&&\\
  \hline
  &&&&\\
  $F_4$ & $\begin{array}{l}
             1 \\
             \circ \hspace{-1.3mm}- \hspace{-1.3mm}\circ \hspace{-1.4mm}=\hspace{-1.4mm}  \circ\hspace{-1.3mm} -\hspace{-1.3mm} \circ
           \end{array}
  $& $ n=26 $& 3 & $m_0=m_1=m_2=8$ \\
  &&&&\\
  \hline
  &&&&\\
  $S\mathcal{O}(5)$ & $Ad$& $ n=10$ & 4 & $m_0=m_1=m_2=m_3=2$ \\
  &&&&\\
  \hline
  &&&&\\
  $S\mathcal{O}(2)\times S\mathcal{O}(m)$ &$\rho_2 \otimes \rho_m $& $ n=2m \,\geq 6$ & 4 & $m_0=m_2=(m-2), \, m_1=m_3=1$ \\
  &&&&\\
  \hline
  &&&&\\
  $S \left ( U(2)\times U(m) \right )$ &$[\mu_2 \otimes _{\C}\mu_m]_{\R}$ & $ n=4m \,\geq 8$ & 4 & $m_0=m_2=(2m-3), \, m_1=m_3=2$ \\
  &&&&\\
  \hline
  &&&&\\
  $Sp(2) \times Sp(m)$&$ \nu_2 \otimes _{\H}\nu_m^*$& $ n=8m \,\geq 16$ & 4 & $m_0=m_2=(4m-5), \, m_1=m_3=4$ \\
  &&&&\\
  \hline
  &&&& \\
 $U(5)$ &$[\Lambda^2 \mu_5]_{\R}$&$n=20$&$4$&$m_0=m_2=5, \, m_1=m_3=4$\\
 &&&& \\
  \hline
  &&&&\\
 $ U(1)\times Spin(10)$ &$[\mu_1 \otimes _{\C}\Delta_1^+]_{\R}$ & $ n=32$ & 4 & $m_0=m_2=9, \, m_1=m_3=6$ \\
  &&&&\\
  \hline
  &&&&\\
  $G_2$ & $Ad$& $ n=14$ & 6 & $m_0=m_1= \dots =m_5=2$ \\
  &&&&\\
  \hline
  &&&&\\
  $S\mathcal{O}(4)$ &$\begin{array}{l}
              1 \hspace{1.3mm}3\\
             \circ\hspace{-1.3mm} - \hspace{-1.3mm} \circ
           \end{array} $ & $ n=8$ & 6 & $m_0=m_1= \dots =m_5=1$ \\
  &&&&\\
  \hline
  \end{tabular}
  }
  $$
  \,
  $$
  \caption{Cohomogeneity two $G-$actions on $\R^n$ (see \cite{Hsiang,HL}) (\textbf{Note}: the volume function is given in \eqref{genericafunzionevolume}, the number of distint principal curvatures of $ \Sigma_{\gamma}$ is $(d+1)$ ).}\label{table}
\end{table}

Next, we observe that any cohomogeneity one $G-$invariant hypersurface into $\R^n$ can be described by means of what we call its \emph{profile curve} $\bar{\varphi}=\gamma(s)=(x(s),\,y(s))$ in the orbit space $Q$. More precisely, the $G-$invariant hypersurface corresponding to a profile curve $\gamma$ is $M=\Sigma_{\gamma}= \pi_2^{-1}(\gamma)$. For convenience, we shall now assume that
\begin{equation}\label{sascissacurvilinea}
    \dot{x}^2+\dot{y}^2 =1 \,\, ,
\end{equation}
and also, to fix orientation, that the unit normal $\eta$ to the hypersurface $\Sigma_{\gamma}$ projects down in $Q$ to
\begin{equation}\label{unitnormal}
    d\pi_2\,(\eta)=\nu =-\,\dot{y}\,\frac{\partial}{\partial x}+ \dot{x}\,\frac{\partial}{\partial y}\,\, .
\end{equation}

Now, suppose that $\Sigma_{\gamma}$ is a $G-$invariant hypersurface into $\R^n$ of type $d$ ($d=1,\,2,\,3,\,4$ or $6$), with associated volume function given by \eqref{genericafunzionevolume}. Then $\Sigma_{\gamma}$ possesses $(d+1)$ distinct principal curvatures given by:

\begin{eqnarray}\label{curvatureprincipali}
    k_i &=&-\, \frac{1}{2} \, \frac{d}{d \, \nu}\, \ln \left[w_{(d,i)}(x,y) \right]^2\nonumber\\
    &=& -\, \frac{1}{2}\,  \nu\left(\ln \left[w_{(d,i)}(x,y) \right]^2\right)\,\, , \quad i=0,\, \ldots,\, (d-1) \,\,,
\end{eqnarray}
each of them with multiplicity equal to $m_i$, and
\begin{equation}\label{ultimacurvaturaprincipale}
    k_d =  \ddot{y}\,\dot {x}\,-\,\ddot{x}\,\dot {y} \,\,
\end{equation}
with multiplicity equal to one. We also observe that, using \eqref{definizionew_{(d,i)}} and \eqref{unitnormal}, \eqref{curvatureprincipali} gives
\begin{eqnarray}\label{curvatureprincipali-bis}
    k_i &=& \frac{\dot{y}  \sin (i\,\pi \slash \,d)+\dot{x}  \cos (i\,\pi \slash \,d)}{w_{(d,i)}(x,y)}\nonumber\\
    &=&\frac{w_{(d,i)}(\dot{y},-\dot{x})}{w_{(d,i)}(x,y)}\,\, , \quad i=0,\, \ldots,\, (d-1) \,\,.
\end{eqnarray}
In particular, it follows that the terms $f$ and $|A|^2$ in \eqref{eq: caract_bih_normal} and \eqref{eq: caract_bih_tangent} are given by
\begin{equation}\label{espressionedif}
   f=(\, \ddot{y}\,\dot {x}\,-\,\ddot{x}\,\dot {y}\,)+\sum_{i=0}^{(d-1)}\,m_i\, k_i
\end{equation}
and
\begin{equation}\label{espressionediA^2}
    |A|^2 = (\, \ddot{y}\,\dot {x}\,-\,\ddot{x}\,\dot {y}\,)^2+ \sum_{i=0}^{(d-1)}\,m_i\, (k_i)^2
\end{equation}
respectively, where the explicit expressions for the $k_i$'s are those given in \eqref{curvatureprincipali-bis}. We shall need to compute the gradient and the laplacian of $f$: to this purpose, we work by using on $\Sigma_{\gamma}$ a local system of coordinates of type $\{u_1,\ldots,u_{n-2},s\}$, where   $u=\{u_1,\ldots,u_{n-2}\}$ are local coordinates of a principal orbit. In particular, we observe that, with respect to these local coordinates, the induced metric $g$ satisfies:
\begin{equation}\label{inducedmetric1}
  \det g = \psi(u)\, V^2(x(s),y(s))\,,
\end{equation}
where $\psi$ is a positive function on the principal orbit and
 \begin{equation}\label{inducedmetric2}
  g_{(n-1),k}=\delta_{(n-1),k}\,,\quad k=1,\ldots,\,(n-1)\,\, .
\end{equation}
Now, since $f$ depends only on $s$, it follows immediately that
\begin{equation}\label{gradientef}
\grad f = \dot {f}(s) \,\frac{\partial}{\partial s} \,\,.
\end{equation}
Next, by using \eqref{inducedmetric1} and \eqref{inducedmetric2} in
$$
\Delta f = -\,\frac{1}{\sqrt{\det g}} \,\frac{\partial}{\partial v_i}\left( g^{ij}\,\sqrt{\det g}\, \frac{\partial f}{\partial v_j}\right)\,,\quad v_i=u_i, i=1,\ldots,\,(n-2),\;\; v_{(n-1)}=s\,\,,
$$
we obtain
\begin{equation}\label{gradienteelaplaciano}
 \Delta f =- \ddot{f}- \frac{1}{2}\,\left(\frac{d}{d \, s}\, \ln V^2\right ) \, \dot{f}\,\, .
\end{equation}
We can summarize this discussion in the following proposition which, taking into account that the Ricci tensor field of $\R^n$ vanishes, follows by direct substitution of \eqref{espressionedif}, \eqref{espressionediA^2}, \eqref{gradientef} and \eqref{gradienteelaplaciano} into \eqref{eq: caract_bih_normal} and \eqref{eq: caract_bih_tangent}:
\begin{proposition}\label{equazioniridottedibiarmonicity} Let $\Sigma_{\gamma}$ be a $G-$invariant hypersurface into $\R^n$ of type $d$ ($d=1,\,2,\,3,\,4$ or $6$), with associated volume function given by \eqref{genericafunzionevolume}. Then $\Sigma_{\gamma}$ is  biharmonic  if and only if
 \begin{equation}\label{bitensionecomponentenormale}
  \ddot{f}+ \frac{1}{2}\,\left(\frac{d}{d \, s}\, \ln V^2\right ) \, \dot{f}\, - \Big[(\, \ddot{y}\,\dot {x}\,-\,\ddot{x}\,\dot {y}\,)^2+ \sum_{i=0}^{(d-1)}\,m_i\, (k_i)^2\Big]\,f =0
 \end{equation}
and
\begin{equation}\label{bitensionecomponentetangente}
    \dot{f}\, (f+2\,(\, \ddot{y}\,\dot {x}\,-\,\ddot{x}\,\dot {y}\,)) =0 \,\, .
 \end{equation}
 \end{proposition}

 \begin{remark}\label{re:linearprifile} Here we consider briefly the case where the profile curve $\gamma(s)=(x(s),y(s))$ in $Q$ satisfies $y=m x$, $m\in\R$. Therefore, we can assume that $\gamma$ is parametrized by
\begin{equation}\label{parametrizzazione}
\gamma (s)= \big( \,s \, \cos \sigma,\, s \, \sin \sigma\, \big ) \,\,,
\end{equation}
where $\sigma$ is a constant in the interval $(0,\,\pi \slash d)$. In this case, we have
\begin{equation}\label{condizioneminimalita}
    f=-\,\frac{1}{s}\,\,\sum_{i=0}^{(d-1)} \, m_i \, \cot \left (\sigma -\,\frac{i\pi}{d} \right ) \,\, .
\end{equation}
In particular, by using \eqref{parametrizzazione} and \eqref{condizioneminimalita} into \eqref{bitensionecomponentetangente}, it is immediate to deduce that a curve of this type gives rise to a solution if and only if $f \equiv 0$. In other words, for this family of curves, the associated $\Sigma_{\gamma}$ is biharmonic if and only if it is minimal.
\end{remark}

 \begin{example}\label{esempio} Equations \eqref{bitensionecomponentenormale} and \eqref{bitensionecomponentetangente} express, respectively, the vanishing of the normal and of the tangential component of the bitension field. In order to help the reader, we now compute explicitly the various principal curvatures and the biharmonicity conditions in one specific instance. More precisely, suppose that $\Sigma_{\gamma}$ is a $G-$invariant hypersurface into $\R^n$ of type $d=4$, with $G=U(5)$, $n=20$ (see Table~\ref{table}). In this case we have:
 $$
 w_{(4,0)}(x,y)= -y \,\, , \quad w_{(4,1)}(x,y)= \frac{\sqrt 2}{2}\,(x-y)\, ,
 $$
 $$
w_{(4,2)}(x,y)= x \,\, , \quad w_{(4,3)}(x,y)= \frac{\sqrt 2}{2}\,(x+y)\,,
 $$
 with $m_0=m_2=5$, $m_1=m_3=4$ (note that $1+\sum_{i=0}^{(d-1)}m_i= \dim (\Sigma_\gamma)=(n-1)=19$). Then we have:
$$
 \begin{aligned}
  k_0& =-\, \frac{1}{2} \, \frac{d}{d \, \nu}\, \ln \left[w_{(4,0)}(x,y)\right ]^2 \\
 &= -\, \frac{1}{w_{(4,0)}(x,y)} \left (- \dot{y}\,\frac{\partial}{\partial x}w_{(4,0)}(x,y) + \dot{x}\,\frac{\partial}{\partial y}w_{(4,0)}(x,y)\right ) = -\,\frac{\dot x}{y} \,\, ,
\end{aligned}
$$
 and, similarly,
 $$
  k_1=\frac{\dot{x}+\dot{y}}{x-y},\,\, \quad k_2= \frac{\dot y}{x}, \,\, \quad k_3=\frac{-\dot{x}+\dot{y}}{x+y} \,\, .
 $$
 Moreover, up to an irrelevant multiplicative constant,
 $$
  V^2(x,y)= \prod_{i=0}^{3}\,  \left [\,w_{(4,i)}(x,y)\,\right ]^{2\,m_i}= (xy)^{10} \, (x^2-y^2)^8 \,\, .
 $$
 Therefore, taking into account  \eqref{espressionedif}, \eqref{espressionediA^2} and \eqref{gradienteelaplaciano} the biharmonicity equations \eqref{bitensionecomponentenormale} and \eqref{bitensionecomponentetangente} become
 \begin{equation}\label{bitensionecomponentenormaleesempio}
  \ddot{f}+\,\left( 5\,\frac{\dot{x}}{x}+5\,\frac{\dot{y}}{y}+4\,\frac{\dot{x}+\dot{y}}{x+y}+4\,\frac{\dot{x}-\dot{y}}{x-y} \right ) \, \dot{f}\, - |A|^2\,f =0
 \end{equation}
and
 \begin{equation}\label{bitensionecomponentetangenteesempio}
    \dot{f}\, (f+2\,(\ddot{y}\,\dot {x}\,-\,\ddot{x}\,\dot {y})) =0
 \end{equation}
 respectively, where
 \begin{equation}\label{f-esempio}
 f= (\ddot{y}\,\dot {x}\,-\,\ddot{x}\,\dot {y})-5\, \frac{\dot x}{y}+4\,\frac{(\dot{x}+\dot{y})}{(x-y)}+5\, \frac{\dot y}{x}+4\,\frac{(-\dot{x}+\dot{y})}{(x+y)}
 \end{equation}
 and
 \begin{equation}\label{A^2esempio}
 |A|^2= (\ddot{y}\,\dot {x}\,-\,\ddot{x}\,\dot {y})^2+\,5\, \left(\frac{\dot x}{y}\right)^2+4\,\left(\frac{\dot{x}+\dot{y}}{x-y}\right)^2+5\, \left(\frac{\dot y}{x}\right)^2 +4 \,\left( \frac{-\dot{x}+\dot{y}}{x+y}\right )^2 \,\,.
 \end{equation}
 \end{example}

Proposition~\ref{equazioniridottedibiarmonicity} above was derived by computing directly the various terms in the two equations of Theorem~\ref{th: bih subm N} by means of the equivariant geometric results of \cite{Pedrosa}, \cite{BackDoCarmoHsiang}. This approach, which was followed in \cite{MOR-G-Invariant-biharm-immersions}, is of interest because it displays explicitly the principal curvatures and the structure of the second fundamental form of $\Sigma_{\gamma}$. The end of the proof of Theorem~\ref{Main-theorem} can be found in \cite{MOR-G-Invariant-biharm-immersions} and
so we do not include it here. We limit ourselves to say that it uses an idea introduced in \cite{HasVla95}: essentially, by replacing the information given by equation \eqref{bitensionecomponentetangente} into \eqref{bitensionecomponentenormale}, we find that assuming that $\gamma$ produces a local, proper biharmonic solution leads us to contradicting a result in algebraic geometry concerning the \emph{resultant} of a family of homogeneous polynomial.
It is worth to point out that, despite the complexity of the calculations involved, this argument covers together all the cases $d=1,\,2,\,3,\,4$ and $6$. A key point for this purpose is the fact that $G-$invariance is sufficient to guarantee that the biharmonicity conditions \eqref{bitensionecomponentenormale} and \eqref{bitensionecomponentetangente} can be riformulated in a suitable form which displays coefficients all given by certain \emph{homogeneous} polynomials.

\begin{remark} The proof of Theorem~\ref{Main-theorem} is of a local nature. However, we point out that in \cite{Maeta} and \cite{NUG} the authors obtained some global nonexistence results for proper biharmonic complete submanifolds.
\end{remark}

Now, in the spirit of the reduction theory of the present paper, we find it useful to provide here an \textbf{alternative proof of Proposition~\ref{equazioniridottedibiarmonicity}}, that is an alternative way of computing the tangent and normal components of the bitension field of a $G-$invariant hypersurface into $\R^n$, based on the development of the ideas of Remark~\ref{remark-equivariance-per-immersions}. In this context, the relevant calculations can be conveniently  carried out by using a suitable software for symbolic computations.
\begin{proof}
First, we give the proof of Proposition~\ref{equazioniridottedibiarmonicity} in the instance $d=2$. After this, we shall indicate the modifications which are necessary to cover the remaining cases. So, let us assume that $d=2$: then, according to Table~\ref{table}, $G=S\mathcal{O}(p)\times S\mathcal{O}(q)$ and $\Sigma_{\gamma}$ is defined by the isometric immersion:
$$
\varphi_{\gamma}:\s^{p-1}\times \s^{q-1}\times I \hookrightarrow \R^n\,, \quad \varphi_{\gamma}(v,w,s)=(x(s) v, y(s) w)\,.
$$
Locally, assuming that the curve $\gamma$ is in the interior of the orbit space $Q$, $\Sigma_{\gamma}$ admits an orthonormal frame of the following type:
\begin{equation}\label{orthonormal-frame-su-sigma-gamma-d=2}
   \left \{ \,\,\frac{1}{x(s)}\,v_1,\, \dots,\, \frac{1}{x(s)}\,v_{p-1},\,\,\, \frac{1}{y(s)}\,w_1,\, \dots,\, \frac{1}{y(s)}\,w_{q-1},\, \,\,\frac{1}{\sqrt{\dot{x}^2(s)+\dot{y}^2(s)}}\,
    \frac{\partial}{\partial s} \,\, \right \}\,\, ,
\end{equation}
where
\begin{equation}\label{orthonormal-frame-su-orbita-d=2}
  \left \{\,\,  v_1,\, \dots,\, v_{p-1},\,\, w_1,\, \dots,\, w_{q-1}\,\, \right \}
\end{equation}
is a local orthonormal frame on $\s^{p-1} \times \s^{q-1}$. Now, the key idea is the following: we fix $\Sigma_{\gamma_0}$ defined by a given curve $\gamma_0=[x_0(s),\,y_0(s)]$ which satisfies the additional requirement
\begin{equation}\label{ascissa-curvilinea-gamma-zero}
\dot{x}_0^2(s)+\dot{y}_0^2(s) \equiv 1 \,\, .
\end{equation}
In particular, note that we fix on the domain $\s^{p-1}\times \s^{q-1}\times I$ the induced metric by $\varphi_{\gamma_0}$ which we call $g_0$: modifying \eqref{orthonormal-frame-su-sigma-gamma-d=2}, a local orthonormal frame with respect to $g_0$ is of the form
\begin{equation}\label{orthonormal-frame-su-sigma-gamma-zero-d=2}
   \left \{ \,\,\frac{1}{x_0(s)}\,v_1,\, \dots,\, \frac{1}{x_0(s)}\,v_{p-1},\,\,\, \frac{1}{y_0(s)}\,w_1,\, \dots,\, \frac{1}{y_0(s)}\,w_{q-1},\, \,
    \frac{\partial}{\partial s} \,\, \right \}\,\, .
\end{equation}
Now, by using the reduction Lemma~\ref{adaptationofXin} and a version of Remark~\ref{remark-relazione-tension-volume}, we want to determine the tension field $\tau(\varphi_{\gamma})$ of a $G-$invariant map  $\varphi_{\gamma}\, : \,(\s^{p-1}\times \s^{q-1}\times I,\,g_0)\,\rightarrow \, \R^n$. After this, by setting $\gamma(s)=\gamma_0(s)$ in the expression of $\tau(\varphi_{\gamma})$, it will be easy to deduce that a $G-$invariant submanifold $\Sigma_{\gamma}$ is minimal if and only if its profile curve $\gamma(s)$ is a solution of \eqref{condizio-armonicity-zero} below. Once this will be done, we shall proceed in a similarly fashion for the biharmonic case. However, let us now proceed step by step. First, by a simple rescaling argument which uses \eqref{orthonormal-frame-su-sigma-gamma-d=2} and \eqref{orthonormal-frame-su-sigma-gamma-zero-d=2}, we observe that the reduced energy which controls the harmonicity of a $G-$invariant map $\varphi_{\gamma}\, : \,(\s^{p-1}\times \s^{q-1}\times I,\,g_0)\,\rightarrow \, \R^n$ is:
\begin{equation}\label{energiaridotta-immersioni-d02}
    E(\varphi_\gamma)=\omega_0 \int_a^b \,\, L_{\gamma}(s,\,x, \,\dot{x},\,y,\,\dot{y}) \, ds \,\, ,
\end{equation}
where $\omega_0=\rm{Vol}(\s^{p-1})\, \rm{Vol}(\s^{q-1})$ and the Lagrangian $L_{\gamma}$ is the product of the energy density and the volume term as follows:
\begin{equation}\label{lagrangiana-decisiva-d=2}
    L_{\gamma}(s,\,x, \,\dot{x},\,y,\,\dot{y}) \,\,= \,\, e(\varphi_\gamma)\, \, V(s) \,\,,
\end{equation}
where
\begin{equation}\label{energy-decisiva-d=2}
     e(\varphi_\gamma)\, =\, \frac{1}{2} \, \left [\,\dot{x}^2(s)+\dot{y}^2(s)\,+\,(p-1)\, \frac{x^2(s)}{x_0^2(s)}\,+ \,(q-1)\, \frac{y^2(s)}{y_0^2(s)}\,\right ]
\end{equation}
and
\begin{equation}\label{volume-zero-d=2}
     V(s) \,=\, x_0^{\,\,\,p-1}(s)\,\,y_0^{\,\,\,q-1}(s)
\end{equation}
(note that $x_0(s)$ and $y_0(s)$ in \eqref{energy-decisiva-d=2} and \eqref{volume-zero-d=2} are to be considered as preassigned functions).
Next, we compute the tension field $\tau(\varphi_\gamma)$ with respect to the fixed metric $g_0$. For convenience, we set
\begin{equation}\label{definizione-di-A-x}
    A_{x}\, (s,\,x, \,\dot{x},\,y,\,\dot{y})\,\,=\, \,\frac{\partial L_{\gamma}}{\partial x}- \frac{d}{ds} \,  \frac{\partial L_{\gamma}}{\partial \dot{x}}
\end{equation}
and
\begin{equation}\label{definizione-di-A-y}
    A_{y}\, (s,\,x, \,\dot{x},\,y,\,\dot{y})\,\,=\, \, \frac{\partial L_{\gamma}}{\partial y}-\frac{d}{ds} \,  \frac{\partial L_{\gamma}}{\partial \dot{y}}\,\, .
\end{equation}
Now, we calculate the directional derivative of the energy functional  (compare with Remark~\ref{remark-relazione-tension-volume}).

Let $\beta:I\to \R^2$, $\beta(s)=(\beta_x(s),\beta_y(s))$,  be a smooth curve that vanishes outside $(a,b)$. We have
\begin{eqnarray}\label{equation-cezar1}
\frac{d}{dh} \, \left [  E(\,\varphi_{(\gamma \, +\, h\, \beta) }\,)\,\right ] _{|_{h=0}}&=&\omega_0  \int_a^b  
\frac{d}{dh}\left ( L_{\gamma}(s, x+h\beta_x, \dot{x}+h \dot{\beta}_x,  y+h\beta_y, \dot{y}+h \dot{\beta}_y\right )\Big {|}_{h=0}ds\nonumber\\
&=&  \omega_0 \int_a^b \left[\left(\frac{\partial L_{\gamma}}{\partial x}-\frac{d}{ds} \,  \frac{\partial L_{\gamma}}{\partial \dot{x}} \right) \beta_x +  \left(\frac{\partial L_{\gamma}}{\partial y}-\frac{d}{ds} \,  \frac{\partial L_{\gamma}}{\partial \dot{y}} \right) \beta_y \right] ds\,.\nonumber\\
\end{eqnarray}
On the other hand, the variational vector field along $\varphi_{\gamma}$ corresponding to the variation $\{\varphi_{\gamma+h \beta}\}_h$ is 
$$
{\mathcal V}= \beta_x \frac{\partial}{\partial x}+\beta_y \frac{\partial}{\partial y}\,.
$$
Thus 
\begin{eqnarray}\label{equation-cezar2}
\frac{d}{dh} \, \left [  E(\,\varphi_{(\gamma \, +\, h\, \beta) }\,)\,\right ] _{|_{h=0}}&=&-\int_{\Sigma_{\gamma_0}} \langle \tau(\varphi_{\gamma}), {\mathcal V} \rangle\,  v_{g_0}\nonumber\\
&=& -\omega_0 \int_a^b \left( \tau_x \beta_x+\tau_y \beta_y \right) V(s) \, ds\,,
\end{eqnarray}
where 
\begin{equation}\label{relazione-tra-tau-e-eq-eulero-per-immersioni}
    \tau (\varphi_{\gamma})= \tau_{x}\,\,\, \frac{\partial}{\partial x} \,\,+\,\,\tau_{y} \,\, \frac{\partial}{\partial y} \,\, .
\end{equation}
Comparing  \eqref{equation-cezar1} and \eqref{equation-cezar2}, using \eqref{definizione-di-A-x} and \eqref{definizione-di-A-y}, we deduce that

\begin{equation}\label{tau-x}
   \tau_{x}\,\,(s,\,x, \,\dot{x},\,y,\,\dot{y})\,\,=-\,\frac{A_{x}\,\,(s,\,x, \,\dot{x},\,y,\,\dot{y})\,}{V(s)}
\end{equation}
and
\begin{equation}\label{tau-y}
   \tau_{y}\,\,(s,\,x, \,\dot{x},\,y,\,\dot{y})\,\,=-\,\frac{A_{y}\,\,(s,\,x, \,\dot{x},\,y,\,\dot{y})\,}{V(s)} \,\, .
\end{equation}

A computation shows that the explicit expressions of \eqref{tau-x} and \eqref{tau-y} are:
\begin{equation}\label{tau-x-esplicita}
  \tau_{x}=\ddot{x}(s)+(p-1)\frac{\dot{x}(s)
   \dot{x}_0(s)}{x_0(s)}
   + (q-1)\frac{\dot{x}(s)
   \dot{y}_0(s)}{y_0(s)}
   -(p-1)\frac{x(s)}
   {x_0(s)^2}    \end{equation}
and
\begin{equation}\label{tau-y-esplicita}
  \tau_{y}=\ddot{y}(s)+(p-1)\frac{\dot{y}(s)
   \dot{x}_0(s)}{x_0(s)}
   + (q-1)\frac{\dot{y}(s)
   \dot{y}_0(s)}{y_0(s)}
   -(q-1)\frac{y(s)}
   {y_0(s)^2}
\end{equation}
Next, as announced above, we can set $x(s)=x_0(s),\,y(s)=y_0(s)$ in \eqref{relazione-tra-tau-e-eq-eulero-per-immersioni} and simplify \eqref{tau-x-esplicita}, \eqref{tau-y-esplicita} taking into account the assumption \eqref{ascissa-curvilinea-gamma-zero}, which now becomes
\begin{equation}\label{ascissa-curvilinea-gamma}
\dot{x}^2(s)+\dot{y}^2(s) \equiv 1 \,\, .
\end{equation}
By way of conclusion, we obtain the following expression for the tension field $\tau (\varphi_{\gamma})$ of an \emph{isometric} immersion:
\begin{equation}\label{tau-immersione-isometrica}
    \tau (\varphi_{\gamma})= \left(\ddot{x}-(p-1)\frac{\dot{y}^2}{x}+(q-1)\frac{\dot{x}\dot{y}}{y}\right) \frac{\partial}{\partial x}+
    \left(\ddot{y}-(q-1)\frac{\dot{x}^2}{y}+(p-1)\frac{\dot{x}\dot{y}}{x}\right) \frac{\partial}{\partial y}
\end{equation}
We observe that it is possible to check that the tangential component of the tension field in \eqref{tau-immersione-isometrica}, i.e.,
\begin{equation}\label{tau-immersione-isometrica-componente-tangente}
    \tau_t (\varphi_{\gamma})=  \tau (\varphi_{\gamma}) \,\, \cdot \,\, \left [\,\, \dot{x}\,\,\frac{\partial}{\partial x}\, \,+\,
    \,\dot{y}\,\,\frac{\partial}{\partial y} \, \right ]
\end{equation}
always vanishes. Moreover, the normal component of the tension field in \eqref{tau-immersione-isometrica}, i.e.,
\begin{equation}\label{tau-immersione-isometrica-componente-normale}
    \tau_{\nu} (\varphi_{\gamma})=  \tau (\varphi_{\gamma}) \,\, \cdot \,\, \left [\, -\,\dot{y}\,\,\frac{\partial}{\partial x}\, \,+\,
    \,\dot{x}\,\,\frac{\partial}{\partial y} \, \right ]
\end{equation}
is given by:
\begin{equation}\label{tau-immersione-isometrica-componente-normale-esplicita}
    \tau_{\nu} (\varphi_{\gamma})= \ddot{y} \dot{x} - \ddot{x} \dot{y} +(p-1) \frac{\dot{y}}{x} - (q-1) \frac{\dot{x}}{y} \,\, .
\end{equation}
In particular, the equation
\begin{equation}\label{condizio-armonicity-zero}
    \tau_{\nu} (\varphi_{\gamma})\,=\,0
\end{equation}
represents the condition under which $\Sigma_{\gamma}$ is minimal. Now, the next step is to extend this machinery to the biharmonic case. To this purpose, we first have to compute the bitension field $\tau_2(\varphi_{\gamma})$ of our $G-$invariant map $\varphi_{\gamma}\, : \,(\s^{p-1}\times \s^{q-1}\times I,\,g_0)\,\rightarrow \, \R^n$. The reduced bienergy which controls the biharmonicity of the $G-$invariant map $\varphi_{\gamma}\, : \,(\s^{p-1}\times \s^{q-1}\times I,\,g_0)\,\rightarrow \, \R^n$ is:
\begin{equation}\label{bienergiaridotta-immersioni-d02}
    E_2(\varphi_\gamma)= \omega _0 \int_a^b \,\, L_{2,\,\gamma}(s,\,x, \,\dot{x},\,\ddot{x},\,y,\,\dot{y},\,\ddot{y}) \, ds \,\, ,
\end{equation}
where the Lagrangian $L_{2,\,\gamma}$ is:
\begin{equation}\label{lagrangiana-decisiva-d=2}
    L_{2,\,\gamma}(s,\,x, \,\dot{x},\,\ddot{x},\,y,\,\dot{y},\,\ddot{y}) = \frac{1}{2} |\,\tau(\varphi_\gamma)\,|^2 \, \, V(s) \,\,,
\end{equation}
where the volume term $V(s)$ is as in \eqref{volume-zero-d=2} and
\begin{equation}\label{norma-tension-field}
    |\,\tau(\varphi_\gamma)\,|^2 = \tau_x^2 \,+\, \tau_y^2
\end{equation}
can be made explicit by using \eqref{tau-x-esplicita} and \eqref{tau-y-esplicita} (note that we are not yet assuming that $\varphi_\gamma$ is an isometric immersion).
Next, for convenience we set:
\begin{equation}\label{definizione-di-B-x}
    B_{x}\, (s,\,x, \,\dot{x},\,\ddot{x},\,y,\,\dot{y},\,\ddot{y})\,\,=\,\,\,\frac{\partial L_{2,\,\gamma}}{\partial x}\,\,- \, \, \frac{d}{ds} \,  \frac{\partial L_{2,\,\gamma}}{\partial \dot{x}}\,+\,\,\frac{d^2}{ds^2}\,\,\frac{\partial L_{2,\,\gamma}}{\partial \ddot{x}}
\end{equation}
and
\begin{equation}\label{definizione-di-B-y}
    B_{y}\, (s,\,x, \,\dot{x},\,\ddot{x},\,y,\,\dot{y},\,\ddot{y})\,\,=\, \,\,\frac{\partial L_{2,\,\gamma}}{\partial y} \,\,-\,\, \frac{d}{ds} \,  \frac{\partial L_{2,\,\gamma}}{\partial \dot{y}}\,\,+\,\,\frac{d^2}{ds^2}\,\,\frac{\partial L_{2,\,\gamma}}{\partial \ddot{y}}\,\, .
\end{equation}
An argument similar to the one which we used for the tension field leads us to the following expression for the bitension field:
\begin{equation}\label{relazione-tra-bitau-e-eq-eulero-per-immersioni}
    \tau_2 (\varphi_{\gamma})= \tau_{2,\,x}\,\,\, \frac{\partial}{\partial x} \,\,+\,\,\tau_{2,\,y} \,\, \frac{\partial}{\partial y} \,\, ,
\end{equation}
where
\begin{equation}\label{tau2-x}
   \tau_{2,\,x}\,\,(s,\,x, \,\dot{x},\,\ddot{x},\,y,\,\dot{y},\,\ddot{y})\,\,=\,\frac{B_{x}\,\,(s,\,x, \,\dot{x},\,\ddot{x},\,y,\,\dot{y},\,\ddot{y})\,}{V(s)}
\end{equation}
and
\begin{equation}\label{tau2-y}
    \tau_{2,\,y}\,\,(s,\,x, \,\dot{x},\,\ddot{x},\,y,\,\dot{y},\,\ddot{y})\,\,=\,\frac{B_{y}\,\,(s,\,x, \,\dot{x},\,\ddot{x},\,y,\,\dot{y},\,\ddot{y})\,}{V(s)} \,\, .
\end{equation}

Now, as above, we can assume that $\varphi_\gamma$ is an \emph{isometric} immersion. Thus, we use $x(s)=x_0(s),\,y(s)=y_0(s)$ in \eqref{relazione-tra-bitau-e-eq-eulero-per-immersioni} and simplify taking into account \eqref{ascissa-curvilinea-gamma} and its derivative, i.e.,
\begin{equation}\label{ascissa-curvilinea-gamma-derivata}
2\,\dot{x}(s)\, \ddot{x}(s)\,+\, 2\,\dot{y}(s)\, \ddot{y}(s)\, \equiv \, 0 \,\, .
\end{equation}
After a computation (using a software for symbolic computations)
it is possible to express the tangential component of the bitension field, i.e. :
\begin{equation}\label{bitau-immersione-isometrica-componente-tangente}
    \tau_{2,\,t} (\varphi_{\gamma})=  \tau_2 (\varphi_{\gamma}) \,\, \cdot \,\, \left [\,\, \dot{x}\,\,\frac{\partial}{\partial x}\, \,+\,
    \,\dot{y}\,\,\frac{\partial}{\partial y} \, \right ] \,\, ,
\end{equation}
In a similar way, the normal component of the tension field can be computed by means of:
\begin{equation}\label{bitau-immersione-isometrica-componente-normale}
    \tau_{2,\,\nu} (\varphi_{\gamma})=  \tau_2 (\varphi_{\gamma}) \,\, \cdot \,\, \left [\, -\,\dot{y}\,\,\frac{\partial}{\partial x}\, \,+\,
    \,\dot{x}\,\,\frac{\partial}{\partial y} \, \right ] \,\, .
\end{equation}
The two explicit results are:
\begin{eqnarray}\label{bitau-immersione-isometrica-componente-tangente-esplicita}
    \tau_{2,\,t} (\varphi_{\gamma})&=&(p-1)\frac{x^{(3)}}{x}+(q-1)\frac{y^{(3)}}{y}+2(p-1)\frac{{\dot{x}}^2 x^{(3)}}{x}+2(q-1)\frac{\dot{y}^2 y^{(3)}}{y} \nonumber\\&& +2(q-1)\frac{\dot{x} \dot{y} x^{(3)}}{y}+2(p-1)\frac{\dot{x}\dot{y} y^{(3)}}{x}  -3 \ddot{x} x^{(3)} -3 \ddot{y} y^{(3)}\nonumber\\&&
-\frac{(p + q - p q - 1)}{x y} \left(\dot{x} \ddot{y}+\ddot{x} \dot{y} \right)-(p^2 - 5 p + 4)\frac{\dot{y}\ddot{y}}{x^2}+(q^2 - 5 q + 4)\frac{\dot{y}\ddot{y}}{y^2}\nonumber\\&&
+\frac{(p + q - p q - 1)\dot{x}\dot{y}}{xy}\left(\frac{\dot{x}}{x}+\frac{\dot{y}}{y}\right)+(p - 1)^2\frac{\dot{x}\dot{y}^2}{x^3}+(q - 1)^2\frac{\dot{x}^2\dot{y}}{y^3}
   \end{eqnarray}
and
\begin{eqnarray}\label{bitau-immersione-isometrica-componente-normale-esplicita}
    \tau_{2,\,\nu} (\varphi_{\gamma})&=&  \dot{x} y^{(4)}- \dot{y} x^{(4)} +2(p-1)\frac{y^{(3)}}{x}-2(q-1)\frac{x^{(3)}}{y}
     \nonumber\\&&
     2 (p q - p - q + 1)\frac{ \dot{y} \ddot{y}}{x y}+(p^2 - 4 p + 3)\frac{ \dot{x} \ddot{y}}{x^2}-(q^2 - 4 q + 3) \frac{ \dot{y} \ddot{x}}{y^2}
     \nonumber\\&&
    + (q-1)^2\frac{\dot{x}}{y^3}- (p-1)^2\frac{\dot{y}}{x^3} -2(q-1)\frac{\dot{x}\dot{y}^2}{y^3}+2(p-1)\frac{\dot{y}\dot{x}^2}{x^3}\,.
  \end{eqnarray}
  Now a direct check shows that \eqref{bitau-immersione-isometrica-componente-tangente-esplicita} can be written, taking into account \eqref{ascissa-curvilinea-gamma} and its derivative, as
$$
  \tau_{2,\,t} (\varphi_{\gamma})=- \frac{d}{ds}\left[\ddot{y} \dot{x} - \ddot{x} \dot{y} +(p-1) \frac{\dot{y}}{x} - (q-1) \frac{\dot{x}}{y}\right] \left(3\ddot{y} \dot{x} -3 \ddot{x} \dot{y} +(p-1) \frac{\dot{y}}{x} - (q-1) \frac{\dot{x}}{y}\right)
$$
which is, once equaled to zero,  equivalent to \eqref{bitensionecomponentetangente}.
In a similar way, \eqref{bitau-immersione-isometrica-componente-normale-esplicita} can be compared with  \eqref{bitensionecomponentenormale}. This ends the proof of Proposition~\ref{equazioniridottedibiarmonicity} in the case $d=2$. \\

In the general case the calculations are exactly the same, but they just have to start with the following energy densities and volume functions which include \eqref{energy-decisiva-d=2} and \eqref{volume-zero-d=2} as special cases:
\begin{equation}\label{energy-decisiva-d-generale}
     e(\varphi_\gamma)\, =\, \frac{1}{2} \, \left [\,\dot{x}^2(s)+\dot{y}^2(s)\,\,+\,\, \sum_{i=0}^{d-1}\,\,m_i\,\, \frac{[w_{(d,i)}\,(x(s),\,y(s))]^2} {[w_{(d,i)}\,(x_0(s),\,y_0(s))]^2}\,\,\right ]
\end{equation}
and
\begin{equation}\label{volume-zero-d-generale}
     V(s) \,=\, \sqrt { \prod_{i=0}^{(d-1)}\,  \left [\,w_{(d,i)}(x_0(s),y_0(s))\,\right ]^{2\,m_i}  }\,\, ,
\end{equation}
where the functions $w_{(d,i)}(x,y)$ are those defined in \eqref{definizionew_{(d,i)}}
and the corresponding multiplicities $m_i$'s are those given in Table~\ref{table}.

\end{proof}
\begin{remark}\label{software} It is very important to stress the following fact: the previous proof of Proposition~\ref{equazioniridottedibiarmonicity} is geometrically less elegant and requires much longer computations with respect to the original one given in \cite{MOR-G-Invariant-biharm-immersions}. To be more explicit, we can state that a patient reader may carry out by hand the necessary calculations starting from the Lagrangian in \eqref{lagrangiana-decisiva-d=2}. By contrast, the general case, which requires \eqref{energy-decisiva-d-generale} and \eqref{volume-zero-d-generale}, is reasonably possible only by using a suitable computer aided procedure. However, we stress the fact that this approach provides an useful tool because we are able to implement this procedure by means of a reasonably simple computer algorithm: for this reason, we believe that in the future we may be able to treat various other similar, computationally very complicated, related geometric situations.
\end{remark}


\begin{remark}\label{biconservative}
The study of the stress-energy tensor for the
bienergy was initiated  in \cite{GYJ2} and afterwards developed in \cite{LMO}. Its expression is \begin{eqnarray*}
S_2(X,Y)&=&\frac{1}{2}\vert\tau(\varphi)\vert^2\langle X,Y\rangle+
\langle d\varphi,\nabla\tau(\varphi)\rangle \langle X,Y\rangle \\
\nonumber && -\langle d\varphi(X), \nabla_Y\tau(\varphi)\rangle-\langle
d\varphi(Y), \nabla_X\tau(\varphi)\rangle,
\end{eqnarray*}
and it satisfies the condition
\begin{equation}\label{eq:2-stress-condition}
\Div S_2=-\langle\tau_2(\varphi),d\varphi\rangle,
\end{equation}
thus conforming to the principle of a  stress-energy tensor for the
bienergy. If  $\varphi:M\hookrightarrow(N,h)$ is an isometric immersion, then \eqref{eq:2-stress-condition} becomes
$$
(\Div S_2)^{\#}=-\tau_2(\varphi)^{\top}\,,
$$
where $\#$ denotes the musical isomorphism sharp.
This means that isometric immersions with $\Div S_2=0$ correspond to immersions with vanishing tangent part of the corresponding bitension field: such immersions are called \emph{biconservative}. In particular, if we consider a $G-$invariant hypersurface $\Sigma_{\gamma}\, \hookrightarrow \, \R^n$ as in this section, we conclude that $\Sigma_{\gamma}$ is biconservative if and only if \eqref{bitensionecomponentetangente} is satisfied. In the cases $d=1,\,2$ a detailed, qualitative study of such hypersurfaces was carried out in \cite{MOR-AMPA} where, in particular, the existence of proper (not CMC), biconservative cones was proved.
\end{remark}


\section{Open problems and further developments}\label{open-problems}

\begin{problem}\label{generale-esistenza}

 The general existence problem for harmonic maps can be formulated as follows (\cite{ES}): given a homotopy class of maps between two Riemannian manifolds, does it admit a \emph{harmonic} representative? For instance, J. Eells and J.C. Wood proved (see \cite{JEELLE1} and references therein) that there exist no harmonic map $\varphi \, : \, T^2\,\rightarrow \, \s^2$ of degree $1$. Does there exist a biharmonic map $\varphi \, : \, T^2\,\rightarrow \, \s^2$ of degree $1$? More generally: given a homotopy class of maps between two Riemannian manifolds, does it admit a \emph{proper biharmonic} representative?
\end{problem}

\begin{problem}\label{generale-esistenza-Dirichlet}

Problem~\ref{generale-esistenza} admits a natural generalization to the case of manifolds with boundary (Dirichlet's problem). In particular, we pointed out in Remark~\ref{Jager-Kaul} that a certain Dirichlet's problem for maps from the closed, Euclidean unit $4$-ball $B^4$ to $\s^4$ does not admit rotationally symmetric harmonic solutions, while, by contrast, there exist a proper biharmonic solution. In this order of ideas, for instance, we know that any harmonic map $\varphi \, : \, B^m\,\rightarrow \, N$, $m \geq2$, which is constant on the boundary is constant on the whole $B^m$. \emph{Does the same property holds for proper biharmonic maps?} (a partial result in this direction was proved in \cite{Baird-Fardoun-2010}).
\end{problem}

\begin{problem}\label{trave-problem}

In the introduction we pointed out that the motion of a beam can be determined by studying the fourth order equation \eqref{trave}. It would be interesting to develop this fact in a variational, riemannian geometric context and compare the outcome with the theory of biharmonic maps.
\end{problem}

\begin{problem}\label{m=4}

Due to the conformal invariance of the energy functional when the domain manifold is $2$-dimensional, it is well-known that the study of harmonic maps from surfaces is particularly rich and interesting (see \cite{JEELLE1}, \cite{JEELLE2}). The results and examples of Section~\ref{mappe-tra-modelli} suggest that the domain dimension $m=4$ may be special in the study of biharmonic maps: it would be interesting to clarify and develop this idea.
\end{problem}

\begin{problem}\label{G-invarianti-immersioni-in-sfere}

We have a rather complete knowledge of the possible isometric actions on the Euclidean sphere which enable us to produce families of $G-$invariant, cohomogeneity one hypersurfaces. In this context, suitable computer aided algorithms could be used to produce the equations which govern biharmonicity.
\end{problem}

\begin{problem}\label{G-invarianti-immersioni-in-euclideo-cohomogeneity-due}

The Chen conjecture does not concern hypersurfaces only: for instance, one could study $G-$invariant, cohomogeneity two biharmonic submanifolds of $\R^n$. In this context, one may try to apply the reduction methods of this paper.
\end{problem}

\end{document}